\documentclass[11pt,dvipdfmx]{amsart}

\usepackage{amsmath,amssymb}
\usepackage{amsthm}
\usepackage[abbrev]{amsrefs}
\usepackage{latexsym}
\usepackage{txfonts}
\usepackage{graphicx}
\usepackage{tikz}
\usepackage{bm}

\allowdisplaybreaks[1]

\title[A wavelet basis for $C^n$-functions and $n$-th Lipschitz functions]{A wavelet basis for non-Archimedean $C^n$-functions and $n$-th Lipschitz functions}
\date{}
\author{Hiroki Ando and Yu Katagiri}

\newtheorem{thm}{Theorem}[section]
\newtheorem*{thm*}{Theorem}
\newtheorem{lem}[thm]{Lemma}
\newtheorem*{lem*}{Lemma}
\newtheorem{prop}[thm]{Proposition}
\newtheorem*{prop*}{Proposition}
\newtheorem{cor}[thm]{Corollary}
\newtheorem{cor*}{Corollary}

\theoremstyle{definition}
\newtheorem{dfn}[thm]{Definition}
\newtheorem*{dfn*}{Definition}
\newtheorem{ex}[thm]{Example}
\newtheorem*{ex*}{Example}
\newtheorem{rmk}[thm]{Remark}
\newtheorem*{rmk*}{Remark}

\makeatletter
\@addtoreset{equation}{section}

\makeatother

\makeatletter\@namedef{subjclassname@2020}{\textup{2020} Mathematics Subject Classification}\makeatother


\begin{document}
\maketitle

\begin{abstract}
A wavelet basis is a basis for the $K$-Banach space $C(R, K)$ of continuous functions from a complete discrete valuation ring $R$ whose residue field is finite to its quotient field $K$. In this paper, we prove a characterization of $n$-times continuously differentiable functions from $R$ to $K$ by the coefficients with respect to the wavelet basis and give an orthonormal basis for $K$-Banach space $C^n(R, K)$ of $n$-times continuously differentiable functions.
\end{abstract}

\section{Introduction}

Let $K$ be a local field, i.e., the quotient field of a complete discrete valuation ring $R$ whose residue field $\kappa$ is finite of cardinality $q$. One equips $K$ with the non-Archimedean norm $|\cdot|$ normalized so that $|\pi|=q^{-1}$ for a uniformizer $\pi$ of $K$. The $K$-vector space $C(R, K)$ of continuous functions from $R$ to $K$, equipped with the supremum norm $|f|_{\sup}=\sup_{x \in R}\{|f(x)|\}$, is a $K$-Banach space. 
We employ the following definition for the $C^n$-functions (or $n$-times continuously differentiable functions) as follows.

\begin{dfn}[{\cite[Definition 29.1]{Sc84}}]
For a positive integer $n$, set 
\begin{align}\label{btd}
\bigtriangledown^nR \coloneqq \{ (x_1, \cdots , x_n) \in R^n \mid \text{if} \ i \neq j \ \text{then} \ x_i \neq x_j \}.
\end{align}
The {\it $n$-th difference quotient} $\Phi_nf : \bigtriangledown^{n+1}R \rightarrow K$ of a function $f : R \rightarrow K$ is inductively given by $\Phi_0f \coloneqq f$ and by
\begin{align}\label{n-thquot}
\Phi_nf(x_1, \cdots, x_n, x_{n+1})=\frac{\Phi_{n-1}f(x_1, x_3, \cdots, x_{n+1})-\Phi_{n-1}f(x_2, x_3, \cdots, x_{n+1})}{x_1-x_2}
\end{align}
for $n \in \mathbb{Z}_{>0}$. For $n \geq 0$, a function $f : R \rightarrow K$ is a {\it $C^n$-function} (or an {\it $n$-times continuously differentiable function}) if $\Phi_nf$ can be extended to a continuous function from $R^{n+1}$ to $K$. We denote the set of all $C^n$-functions $R \rightarrow K$ by $C^n(R, K)$ for $n \geq 0$. Note that $C^0(R, K)=C(R, K)$. We define a continuous function $D_nf : R \rightarrow K$ to be $D_nf(x)=\Phi_nf(x, \cdots, x)$ for $f \in C^n(R, K)$.
\end{dfn}

\begin{rmk}\label{rmkPhi}
\begin{enumerate}
\item The $n$-th difference quotient $\Phi_nf$ is a symmetric function of its $n+1$ variables for any $f : R \rightarrow K$ and we have $C^{n+1}(R, K) \subset C^n(R, K)$ for $n \geq 0$ (\cite[Lemma 29.2]{Sc84}).
\item Let $n \geq 1$ and $f \in C(R, K)$. Then, $f \in C^n(R, K)$ if and only if the limit
\begin{align*}
\lim_{\substack{(x_1, \cdots, x_n, x_{n+1}) \to (a, \cdots, a) \\ (x_1, \cdots, x_n, x_{n+1}) \in \bigtriangledown^{n+1}R}} \Phi_nf(x_1, \cdots, x_n, x_{n+1})
\end{align*}
exists for each $a \in R$.
\item For $f \in C(R, K)$, we say that $f$ is ($1$-times) differentiable and write $f^{(1)}=f' : R \rightarrow K$ if the limit
\begin{align*}
f'(a) \coloneqq \lim_{x \to a}\frac{f(x)-f(a)}{x-a}
\end{align*}
exists for any $a \in R$. We define $n$-times differentiable functions inductively as follows. For $n \geq 1$ and an $n$-times differentiable function $f : R \rightarrow K$, we say $f$ is $(n+1)$-times differentiable and write $f^{(n+1)}=(f^{(n)})' : R \rightarrow K$ if the limit
\begin{align*}
f^{(n+1)}(a) \coloneqq \lim_{x \to a}\frac{f^{(n)}(x)-f^{(n)}(a)}{x-a}
\end{align*}
exists for any $a \in R$.
\item Let $n \geq 1$. If $f \in C^n(R, K)$, then $f$ is $n$-times differentiable and 
\begin{align}\label{Dandder}
j!D_jf=f^{(j)}
\end{align}
for any $1 \leq j \leq n$ (\cite[Theorem 29.5]{Sc84}).
\item Let $n \geq 1$ and $f \in C^n(R, K)$. For any $0 \leq j \leq n$, we have $D_{n-j}f \in C^j(R, K)$ and 
\begin{align}\label{propertyD_j}
D_jD_{n-j}f=\binom{n}{j}D_nf
\end{align}
(\cite[Theorem 78.2]{Sc84}).
\item Contrary to the Archimedean case, an $n$-times differentiable function $f : R \rightarrow K$ whose $n$-th derivative $f^{(n)}$ is continuous is not $C^n$ in general. For example, the function defined in \cite[Example 26.6]{Sc84} is not $C^1$. See also \cite[Section 29]{Sc84}.
\item There is another notion of $C^n$-function (see e.g. \cite{BB10}, \cite{Co14} and \cite{Na18}), which we will not discuss in the present paper.
\end{enumerate}
\end{rmk}

It is known that $C^n(R, K)$ is also a $K$-Banach space for each $n \geq 0$, with respect to the norm $|\cdot|_{C^n}$, where 
\begin{align}\label{classical norm}
|f|_{C^n}=\max_{0 \leq j \leq n}\{|\Phi_jf|_{\sup}\}
\end{align}
for $f \in C^n(R, K)$ and $|\Phi_jf|_{\sup}=\sup_{x \in R^{j+1}}\{|\Phi_jf(x)|\}$ for $0 \leq j \leq n$ (\cite[Exercise 29.C]{Sc84}).

We employ the following definition for the orthonormal basis for a $K$-Banach space as follows.

\begin{dfn}[{\cite[Section 50]{Sc84}}]
Let $B$ be a $K$-Banach space whose norm is $||\cdot||$.
\begin{enumerate}
\item For $x, y \in B$, we write $x \perp y$ if $||x|| \leq ||x-\lambda y||$ for any $\lambda \in K$. The orthogonality relation $\perp$ is symmetric.
\item A subset $\{x_1, x_2, \cdots\} \subset B$ is called {\it orthogonal} if $x_i \perp y$ for any $i \geq 1$ and any $y \in \oplus_{j \neq i}Kx_j$. In addition, we say that a subset $\{x_1, x_2, \cdots\} \subset B$ is {\it orthonormal} if $||x_i||=1$ for each $i \geq 1$.
\item A subset $\{x_1, x_2, \cdots\} \subset B$ whose elements are nonzero is called an {\it orthogonal} (resp. {\it orthonormal}) {\it basis} of $B$ if $\{x_1, x_2, \cdots\}$ is orthogonal (resp. orthonormal) set in $B$ and every element $x \in B$ can be expressed as a convergent sum $x=\sum_{n=1}^\infty c_nx_n$ for some sequence $\{c_n\}_{n \geq 1}$ in $K$.
\end{enumerate}
\end{dfn}

\begin{rmk}
Let $B$ be a $K$-Banach space whose norm is $||\cdot||$ and $\{x_1, x_2, \cdots\} \subset B$.
\begin{enumerate}
\item If $\{x_1, x_2, \cdots\}$ is an orthonormal basis of $B$, then $x \in B$ has a unique representation as a convergent sum $x=\sum_{n=1}^\infty c_nx_n$, where $c_n \in K$ and $c_n \to 0$ (\cite[Proposition 50.6]{Sc84}).
\item Suppose that $||x_i||=1$ for all $i \geq 1$. Then $\{x_1, x_2, \cdots\}$ is orthonormal in $B$ if and only if $||\sum_{n=1}^\infty c_nx_n||=\sup_{n \geq 1}\{|c_n|\}$ for each sequence $\{c_n\}_{n \geq 1}$ in $K$ with $c_n \to 0$. This follows from \cite[Proposition 50.4]{Sc84}.
\end{enumerate}
\end{rmk}

Fix a uniformizer $\pi$ of $K$ and let $\mathcal{T}$ be a set of representatives, containing $0 \in R$, of $\kappa$ in $R$. Set
\begin{align*}
\mathcal{R}_m=
\begin{cases}
\{0\} & m=0 \\
\left\{ \displaystyle\sum_{i=0}^{m-1} a_i\pi^i \ \middle | \ a_i \in \mathcal{T} \right\} & m \geq 1,
\end{cases}
\end{align*}
$\mathcal{R} \coloneqq \cup_{m \geq 0}\mathcal{R}_m$ and $\mathcal{R}_+ \coloneqq \cup_{m \geq 1}\mathcal{R}_m$. For $x \in R$, we call the expansion $x=\sum_{i=0}^\infty a_ip^i$ with $a_i \in \mathcal{T}$ ``the $\pi$-adic expansion of $x$'' in this paper. In \cite{dS16}, the following orthonormal basis of $C(R, K)$, which is called the {\it wavelet basis}, was introduced.

\begin{dfn}[{\cite[Section 2]{dS16}}]
Define the {\it length} of $r \in \mathcal{R}$ by
\begin{align}\label{defl(r)}
l(r)=m
\end{align}
where $m$ is such that $r \in \mathcal{R}_m \setminus \mathcal{R}_{m-1}$. The {\it wavelet basis} is defined to be the set of functions $\{\chi_r \mid r \in\mathcal{R}\}$, where $\chi_r$ is the characteristic function of the disk $D_r \coloneqq \left\{x\in R \mid |x-r| \leq |\pi|^{l(r)}\right\}$.
\end{dfn}

\begin{rmk}
\begin{enumerate}
\item For a sequence $\{c_r\}_{r \in \mathcal{R}}$ in $K$, the infinite sum $\sum_{r \in \mathcal{R}}c_r\chi_r$ converges (in $C(R, K)$ with respect to the supremum norm on $R$) if and only if for any $\varepsilon>0$ there exists a finite set $S_{\varepsilon} \subset \mathcal{R}_+$ such that $|c_r|<\varepsilon$ for any $r \in \mathcal{R}_+ \setminus S_{\varepsilon}$.
\item By the same argument as the proof of \cite[Theorem 62.2]{Sc84}, if $f \in C(R, K)$ has the expansion $f=\sum_{r \in \mathcal{R}} b_r(f)\chi_r$, we see that
\begin{align*}
b_r(f)=
\begin{cases}
f(0) & r=0 \\
f(r)-f(r_-) & r \in \mathcal{R}_+.
\end{cases}
\end{align*}
Here, 
\begin{align}\label{defr_-}
r_-=\sum_{i=0}^{m-1}a_i\pi^i
\end{align}
if $r$ has the $\pi$-adic expansion $r=\sum_{i=0}^ma_ip^i$ with $a_m \neq 0$. 
\end{enumerate}
\end{rmk}

To describe our main results, we introduce some notation. Let
\begin{align}\label{defgamma}
\gamma_r=
\begin{cases}
1 & \text{if} \ r=0 \\
r-r_- & \text{if} \ r \in \mathcal{R}_+.
\end{cases}
\end{align}
For $f \in C^n(R, K)$ and $j=0, \cdots, n$, define the continuous function $\psi_jf : \bigtriangledown^2R \rightarrow K$ inductively by $\psi_0f(x, y)\coloneqq \Phi_1f(x, y)$ and
\begin{align}
\psi_jf(x, y)&\coloneqq\frac{f(x)-f(y)-\sum_{l=1}^j(x-y)^lD_lf(y)}{(x-y)^{j+1}} \label{defpsi_jf1} \\
&=\frac{\psi_{j-1}f(x, y)-D_jf(y)}{x-y} \label{defpsi_jf2}
\end{align}
for $j \geq 1$. Note that $\psi_nf(x, y)=\Phi_{n+1}f(x, y, \cdots, y)$. 
One of our main results in this paper is the following.

\begin{thm}\label{main}
Let $n \geq 0$. If $\operatorname{char}(K)=p>0$, we also assume that $n \leq p-1$.
\begin{enumerate}
\item The set $\{\gamma_r^n\chi_r(x), \gamma_r^{n-1}(x-r)\chi_r(x), \cdots, (x-r)^n\chi_r(x) \mid r \in \mathcal{R}\}$ is an orthonormal basis for $C^n(R, K)$. Here, $C^n(R, K)$ is equipped with the supremum norm on $R$ if $n=0$ and the norm $|\cdot|_n$ given by $(\ref{def of norm})$ if $n \geq 1$ (in other words, the norm $|\cdot|_n$ is inductively defined by using the assertion (4) for $n-1$).
\item If $f \in C^n(R, K)$ has the representation $f(x)=\sum_{r \in \mathcal{R}}\sum_{j=0}^n b_r^{n, j}(f)\gamma_r^{n-j}(x-r)^j\chi_r(x)$, then we have
\begin{align*}
b_r^{n, j}(f)=
\begin{cases}
D_jf(0) & \text{if} \ r=0 \\
\gamma_r \psi_{n-j}D_jf(r, r_-) & \text{if} \ r \neq 0
\end{cases}
\end{align*}
for each $0 \leq j \leq n$.
\item Let $f=\sum_{r \in \mathcal{R}}\sum_{j=0}^n b_r^{n, j}(f)\gamma_r^{n-j}(x-r)^j\chi_r \in C^n(R, K)$. Then $f \in C^{n+1}(R, K)$ if and only if the limits $\lim_{\substack{r \to a \\ a \neq r \in \mathcal{R}_+}}b_r^{n, j}(f)\gamma_r^{-1}$ exist for all $a \in R$ and $0 \leq j \leq n$ and satisfy
\begin{align*}
\lim_{\substack{r \to a \\ a \neq r \in \mathcal{R}_+}}b_r^{n, j}(f)\gamma_r^{-1}=\binom{n+1}{j}\lim_{\substack{r \to a \\ a \neq r \in \mathcal{R}_+}}b_r^{n, 0}(f)\gamma_r^{-1}.
\end{align*}
\item Let $f=\sum_{r \in \mathcal{R}}\sum_{j=0}^n b_r^{n, j}(f)\gamma_r^{n-j}(x-r)^j\chi_r \in C^{n+1}(R, K)$. Then 
\begin{align}\label{def of norm}
|f|_{n+1} \coloneqq \sup_{r \in \mathcal{R}}\{|b_r^{n, 0}(f)\gamma_r^{-1}|, \cdots, |b_r^{n, n}(f)\gamma_r^{-1}|\}<\infty 
\end{align}
is a norm on $C^{n+1}(R, K)$. Moreover, $C^{n+1}(R, K)$ is a Banach space over $K$ with respect to the norm $|\cdot|_{n+1}$.
\end{enumerate}
\end{thm}

When $K=\mathbb{Q}_p$, $\pi=p$ and $\mathcal{T}=\{0, 1, \cdots, p-1\}$ for a prime $p$, Theorem \ref{main} for $n=0, 1$ coincide with \cite[Exercise 63.A, Theorem 68.1, Corollary 68.2]{Sc84} and \cite[Theorem 6, Theorem 8, Corollary 9]{DS94}.
Note that the assertions $(1)$ and $(2)$ for $n=0$ already proved in \cite[Section 2]{dS16}. In addition, we will prove the following theorem in Section 4.

\begin{thm}\label{samenorm}
Let $n\geq1$. If ${\rm char}(K)=p>0$, we also assume that $n\leq p-1$. Then $|f|_{n}=|f|_{C^n}$ holds for all $f\in C^n(R,K)$. (See (\ref{classical norm}) and (\ref{def of norm}) for the definition of $|f|_{n}$ and $|f|_{C^n}$.)
\end{thm}

To prove this theorem, we introduce the $n$-th Lipschitz functions.

\begin{dfn}\label{defn-lip}
Let $n\geq1$. A function $f\in C(R,K)$ is called an {\it $n$-th Lipschitz function} if 
\begin{align}\label{dfnn-lip}
\sup\{|\Phi_nf(x_1,\cdots,x_{n+1})|\mid(x_1,\cdots,x_{n+1})\in\bigtriangledown^{n+1}R\}<\infty.
\end{align}
(See (\ref{btd}) for the definition of $\bigtriangledown^{n+1}R$.) We call the value of the left hand side of (\ref{dfnn-lip}) the {\it Lipschitz constant of f} and denote it by $A_f$. We denote the subspace of $n$-th Lipschitz functions by $Lip_n(R,K)$.
\end{dfn}

\begin{rmk}
We define $\Delta_n \subset R^n$ to be
\begin{align}\label{defdelta}
\Delta_n \coloneqq \{(x, \cdots, x) \in R^n\}.
\end{align}
Let $n\geq1$ and $f\in C^{n-1}(R,K)$. Then $\Phi_{n}f$ can be extended to $R^{n+1}\backslash\Delta_{n+1}$ 
and
\begin{align*}
\sup\{|\Phi_nf(x)|\mid x\in\bigtriangledown^{n+1}R\}
=\sup\{\left|\Phi_nf(x)\right|\mid x\in R^{n+1}\backslash\Delta_{n+1}\}.
\end{align*}
In addition, if $f\in C^n(R,K)$, $\Phi_nf$ can be extended to $R^{n+1}$ and
\begin{align*}
\sup\{|\Phi_nf(x)|\mid x\in\bigtriangledown^{n+1}R\}
=\sup\{\left|\Phi_nf(x)\right|\mid x\in R^{n+1}\}.
\end{align*}
These follow from the fact that $R^{n+1}\backslash\Delta_{n+1}$ and $\bigtriangledown^{n+1}R$ are dense in $R^{n+1}$. In the following, we denote the common value by $|\Phi_nf|_{{\rm sup}}$.
\end{rmk}

In fact, an $n$-th Lipschitz function is a $C^{n-1}$-function (see Lemma \ref{inclu}), hence $f$ has the representation
%
\begin{align}\label{expC^{n-1}}
f=\sum_{r\in\mathcal{R}}\sum_{j=0}^{n-1}b_{r}^{n-1,j}(f)\gamma_r^{n-1-j}(x-r)^j\chi_r
\end{align}
by Theorem \ref{main}. The following theorem plays an important role in our proof of Theorem \ref{samenorm}. We note that \cite[Corollary 3.2]{dS16} and \cite[Theorem 63.2]{Sc84} follow as special cases from Theorem \ref{charn-lip}. 

\begin{thm}\label{charn-lip}
Let $n\geq1$ and $f\in C(R,K)$. The following conditions are equivalent.
\begin{enumerate}
\item The function $f$ is an $n$-th Lipschitz function.

\item The function $f$ is a $C^{n-1}$ -function and has the expansion $(\ref{expC^{n-1}})$ with 
\begin{align*}
\sup_{\substack{r\in\mathcal{R}_{+}\\0\leq j\leq n-1}}\left\{\left|b_r^{n-1,j}(f)\gamma_r^{-1}\right|\right\}<\infty.
\end{align*}
\end{enumerate}
Moreover, if these conditions hold, then we have 
\begin{align*}
A_f=\sup_{\substack{r\in\mathcal{R}_{+}\\0\leq j\leq n-1}}\left\{\left|b_r^{n-1,j}(f)\gamma_r^{-1}\right|\right\}.
\end{align*}
\end{thm}



This paper is organized as follows. In Section 2, we will give a proof of Theorem \ref{main} for $n=0$ and introduce $N^1$-functions.
In Section 3, we will prove Theorem \ref{main} by induction on $n$. In Section 4, we will show Theorem \ref{samenorm} and Theorem \ref{charn-lip}. In Section 5, we will introduce monotone, pseudocontraction and isometry functions and give characterizations of those functions by the coefficients with respect to the wavelet basis.

\vspace{10pt}
\noindent
\textsc{Notation:}~ For a field $F$, we denote the characteristic of $F$ by $\operatorname{char}(F)$. Let $K$ be a local field, equipped with the non-Archimedean norm $|\cdot|$ normalized so that $|\pi|=q^{-1}$ for a fixed uniformizer $\pi$ of $K$, and $R$ be the ring of integers of $K$ whose residue field $\kappa$ is finite of cardinality $q$ and $\operatorname{char}(\kappa)=p>0$. We fix a set of representatives $\mathcal{T}$, containing $0 \in R$, of $\kappa$ in $R$. 

\vspace{10pt}
\noindent
\textsc{Acknowledgment:}~ The authors are grateful to our supervisor Professor Takao Yamazaki for his advice and helpful comments. The second-named author is supported by the WISE Program for AI Electronics, Tohoku University.

\section{$C^1$-functions and $N^1$-functions}

\subsection{A preliminary lemma}

Let $x \in R$ and $r \in \mathcal{R}$. We write 
\begin{align}\label{deflhd}
r \lhd x
\end{align}
if $|x-r| \leq q^{-l(r)}$. (See (\ref{defl(r)}) for the definition of $l(r)$.) For example, if $x$ has the $\pi$-adic expansion $x=\sum_{i=0}^\infty a_i\pi^i$, we see that $\sum_{i=0}^{m-1}a_i\pi^i \lhd x$ for any $m \geq 1$. In particular, note that $r_- \lhd r$ for $r \in \mathcal{R}_+$. (See (\ref{defr_-}) for the definition of $r_-$.) To prove the assertion (3) of Theorem \ref{main} for $n=0$, we first show the following key lemma, which is a generalization of \cite[Lemma 63.3]{Sc84}.

\begin{lem}\label{keylem1}
Let $f \in C(R, K)$, let $B$ and $S$ be balls in $R$ and $K$ respectively. If $\Phi_1f(r, r_-) \in S$ for any $r \in \mathcal{R}_+$ with $r, r_- \in B$, then we have $\Phi_1f(x, y) \in S$ for any distinct $x, y \in B$.
\end{lem}

\begin{rmk}\label{rmkconvex}
Let $c \in K$ and $x_1, \cdots x_n \in S=\{x \in K \mid |x-c| < \varepsilon \}$. Then, for $\lambda_1, \cdots, \lambda_n \in K$ satisfying $|\lambda_i| \leq 1$ for each $1 \leq i \leq n$ and $\sum_{i=1}^n\lambda_i=1$, we have $\sum_{i=1}^n\lambda_ix_i \in S$. Indeed, we find that
\begin{align*}
\left|\sum_{i=1}^n\lambda_ix_i-c\right|&=\left|\sum_{i=1}^n\lambda_ix_i-\left(\sum_{i=1}^n\lambda_i\right)c\right| \\
&=\left|\sum_{i=1}^n\lambda_i(x_i-c)\right| \\
&\leq \max_{1 \leq i \leq n}\{|\lambda_i||x_i-c|\} < \varepsilon.
\end{align*}
\end{rmk}

\begin{proof}[Proof of Lemma \ref{keylem1}]
We may assume that $x, y \in B \cap \mathcal{R}$. Indeed, if we suppose that the assertion holds for all pairs of distinct elements in $B \cap \mathcal{R}$, by taking sequences $r_x, r_y \in \mathcal{R}$ with $r_x \to x$ and $r_y \to y$, we see that $\Phi_1f(x, y)=\lim_{(r_x, r_y) \to (x, y)} \Phi_1f(r_x, r_y) \in S$. Note that $\Phi_1f : R^2\setminus \Delta_2 \rightarrow K$ is continuous ($\Delta_2$ is defined in (\ref{defdelta})) and $S$ is closed in $K$.

Let $B=\{x \in R \mid |x-a| < \delta \}$, $S=\{x \in K \mid |x-c| < \varepsilon \}$ and $z$ be the common initial part in the $\pi$-adic expansions of $x$ and $y$, i.e.
\begin{align}\label{defz}
z=
\begin{cases}
\sum_{i=0}^{n-1} a_i \pi^i & \text{if} \ |x-y|=q^{-n}<1, x=\sum_{i=0}^\infty  a_i \pi^i \\
0 & \text{if} \ |x-y|=1.
\end{cases}
\end{align}
By the definition of $z$, we see that $z \lhd x, z \lhd y$ and $|x-y|=\max\{|z-x|, |z-y|\}$. Since 
\begin{align*}
|x-y| \leq \max\{|x-a|, |a-y|\} <\delta,
\end{align*}
we obtain $|x-z|<\delta, |y-z|<\delta$ and
\begin{align*}
|z-a|\leq \max\{|z-x|, |x-a|\}<\delta,
\end{align*}
that is, $z \in B$. Since
\begin{align}\label{lemlem1.1}
\Phi_1f(x, y)=\frac{x-z}{x-y}\Phi_1f(x, z)+\frac{z-y}{x-y}\Phi_1f(z, y),
\end{align}
$|(x-z)/(x-y)|\leq 1, |(z-y)/(x-y)|\leq 1$ and
\begin{align*}
\frac{x-z}{x-y}+\frac{z-y}{x-y}=1,
\end{align*}
according to Remark \ref{rmkconvex}, it suffices to show that $\Phi_1f(x, z) \in S$ and $\Phi_1f(z, y) \in S$. Thus, we may assume that $y \lhd x$ by replacing $z$ with $y$. Then there exists a unique sequence $t_1=y \lhd t_2 \lhd \cdots \lhd t_n=x$ in $\mathcal{R}$ such that $(t_j)_-=t_{j-1}$ for each $2 \leq j \leq n$ and $t_j \in B$ for each $1 \leq j \leq n$. By putting
\begin{align*}
\lambda_j=\frac{t_j-t_{j-1}}{x-y}
\end{align*}
for $2 \leq j \leq n$, we obtain
\begin{align}\label{lemlem2.1}
\Phi_1f(x, y)=\sum_{j=2}^n \lambda_j\Phi_1f(t_j, t_{j-1}),
\end{align}
$\sum_{j=2}^n \lambda_j=1$ and $|\lambda_j| \leq 1$ for each $2 \leq j \leq n$. Since $\Phi_1f(t_j, t_{j-1}) \in S$ for any $2 \leq j\leq n$ by the assumption, Remark \ref{rmkconvex} implies the assertion.
\end{proof}

\subsection{Characterizations of $C^1$-functions and $N^1$-functions}


\begin{thm}\label{mainforn=0}
Let $f \in C(R, K)$ be with the expansion $f=\sum_{r \in \mathcal{R}}b_r(f)\chi_r$. Then, $f \in C^1(R, K)$ if and only if the limit $\lim_{\substack{r \to a \\ a \neq r \in \mathcal{R}_+}} b_r(f)\gamma_r^{-1}$ exists for each $a \in R$.
\end{thm}

\begin{proof}
Suppose that $f$ is a $C^1$-function. Since $\Phi_1f$ is continuous on $R^2$, the limit
\begin{align*}
\lim_{(x, y) \to (a, a)} \Phi_1f(x, y)=D_1f(a) \in K
\end{align*}
exists for any $a \in R$. In other words, for a given $\varepsilon>0$, there exists $\delta>0$ such that $|\Phi_1f(x, y)-D_1f(a)|<\varepsilon$ for any $x, y \in R$ with $|x-a|<\delta$ and $|y-a|<\delta$.

If $a=\sum_{i=0}^\infty a_i\pi^i \notin \mathcal{R}$, we have $a_l \neq 0$ and $q^{-l}<\delta$ for some $l \in \mathbb{Z}_{>0}$. For any $r \in \mathcal{R}_+$ with $|r-a|<q^{-l-1}$, there is $m \geq l$ such that $r=\sum_{i=0}^m a_i\pi^i$ and $a_m \neq 0$. Since $|r-r_-|=|a_mq^m|=q^{-m}\leq q^{-l}$ (see (\ref{defr_-}) for the definition of $r_-$), we obtain
\begin{align*}
|r_--a|\leq \max\{|r_--r|, |r-a|\} \leq q^{-l} <\delta.
\end{align*}
Thus, if $0<|r-a|<q^{-l}$, it follows that $|r_--a|<\delta$.

If $a=\sum_{i=0}^{l(a)-1}a_i\pi^i \in \mathcal{R}$, set $l \coloneqq \min \{i \geq l(a) \mid q^{-i}<\delta\}$, where $l(a)$ was defined in (\ref{defl(r)}). For any $r \in \mathcal{R}_+$ with $0<|r-a|\leq q^{-l}$, since there is $m \geq l$ such that $r-a=\sum_{i=l}^m a_i\pi^i$ and $a_m \neq 0$, we obtain
\begin{align*}
|r_--a|=|r-a_mq^m-a|\leq \max\{q^{-m}, q^{-l}\} \leq q^{-l} <\delta.
\end{align*}
Thus, if $0<|r-a|<q^{-l+1}$, it follows that $|r_--a|<\delta$.

We conclude that, in both cases, there exists $\delta_0>0$ such that 
\begin{align}\label{conclusion}
|\Phi_1f(r, r_-)-D_1f(a)|<\varepsilon
\end{align}
for any $r \in \mathcal{R}_+$ with $0<|r-a|<\delta_0$. Hence, since we have
\begin{align}\label{fracb_rgamma}
b_r(f)\gamma_r^{-1}=\frac{f(r)-f(r_-)}{r-r_-}=\Phi_1f(r, r_-)
\end{align}
for any $r \in \mathcal{R}_+$, the limit $\lim_{\substack{r \to a \\ a \neq r \in \mathcal{R}_+}} b_r(f)\gamma_r^{-1}=D_1f(a)$ exists.

Conversely, we suppose that the limit $\lim_{\substack{r \to a \\ a \neq r \in \mathcal{R}_+}} b_r(f)\gamma_r^{-1} \eqqcolon g(a)$ exists for each $a \in R$. This means that for a given $\varepsilon>0$, there exists $\delta>0$ such that $|\Phi_1f(r, r_-)-g(a)|<\varepsilon$ for any $r \in \mathcal{R}_+$ with $0<|r-a|<\delta$. If $a \notin \mathcal{R}$, Lemma \ref{keylem1} implies that 
\begin{align}\label{conclusion2}
|\Phi_1f(x, y)-g(a)|<\varepsilon
\end{align}
for any $(x, y) \in \bigtriangledown^2R$ with $|x-a|<\delta$ and $|y-a|<\delta$. If $a \in \mathcal{R}$, put $\delta_0 \coloneqq \min\{\delta, q^{-l(a)+1}\}$. If $r \in \mathcal{R}_+$ satisfies $|r-a|<\delta_0$ and $|r_--a|<\delta_0$, we find that $r \neq a$ and $|r-a|<\delta_0 \leq \delta$. Hence, Lemma \ref{keylem1} implies that 
\begin{align}\label{conclusion3}
|\Phi_1f(x, y)-g(a)|<\varepsilon
\end{align}
for any $(x, y) \in \bigtriangledown^2R$ with $|x-a|<\delta_0$ and $|y-a|<\delta_0$. (See (\ref{btd}) for the definition of $\bigtriangledown^2R$.) In either case, we have
\begin{align*}
\lim_{\substack{(x, y) \to (a, a) \\ (x, y) \in \bigtriangledown^2R}} \Phi_1f(x, y)=g(a).
\end{align*}
It follows that $f \in C^1(R, K)$ from \cite[Theorem 29.9]{Sc84}.
\end{proof}

If $f \in C^1(R, K)$ satisfies $f'=0$, $f$ is called an $N^1$-function. We denote the set of all $N^1$-functions by $N^1(R, K)$. Lemma \ref{keylem1} also implies the following theorem.

\begin{thm}\label{gN^1}
Let $f=\sum_{r \in \mathcal{R}}b_r(f)\chi_r \in C(R, K)$. Then, $f \in N^1(R, K)$ if and only if $\lim_{r \in \mathcal{R}_+}b_r(f)\gamma_r^{-1}=0$. Here, for a sequence $\{x_r\}_{r \in \mathcal{R}_+}$ in $K$, we say 
\begin{align}\label{lim_r}
\lim_{r \in \mathcal{R}_+}x_r=x
\end{align}
if for any $\varepsilon>0$ there exists a finite subset $S_{\varepsilon} \subset \mathcal{R}_+$ such that $|x_r-x|<\varepsilon$ for any $r \in \mathcal{R}_+ \setminus S_{\varepsilon}$.
\end{thm}

\begin{proof}
Suppose that $f \in N^1(R, K)$. Thus, there exists a continuous function $\Phi_1f : R^2 \rightarrow K$ satisfying $\Phi_1f(x, x)=0$ for any $x \in R$. Then we see that for any $\varepsilon>0$ there exists $\delta>0$ such that $|\Phi_1f(x, y)|<\varepsilon$ for all $x, y \in R$ with $|x-y|<\delta$. Indeed, since $\Phi_1f(a, a)=0$ for $a \in R$, there is $\delta_a >0$ such that $|\Phi_1f(x, y)|<\varepsilon$ for any $x, y \in R$ with $|x-a|<\delta_a$ and $|y-a|<\delta_a$. Since $\{U_a(\delta_a)\}_{a \in R}$ is an open covering of $\Delta_2$ (see (\ref{defdelta}) for the definition of $\Delta_2$.), where
\begin{align}\label{U_a}
U_a(\delta_a) = \{(x, y) \in R^2 \mid |x-a|<\delta_a, |y-a|<\delta_a\},
\end{align}
and $\Delta_2$ is compact, there exist $a_1, \dots, a_r \in R$ such that $\Delta_2 \subset \cup_{1 \leq j \leq r}U_{a_j}(\delta_{a_j})$. Then we find that $|\Phi_1f(x, y)|<\varepsilon$ if $|x-y|<\delta \coloneqq \min_{1 \leq j \leq r}\{\delta_{a_j}\}$. Put
\begin{align*}
S_\varepsilon \coloneqq \{r \in \mathcal{R}_+ \mid l(r) \leq 1-\log_q \delta \}.
\end{align*}
Since
\begin{align*}
|r-r_-|=q^{-l(r)+1} <\delta
\end{align*}
for any $r \in \mathcal{R}_+\setminus S_\varepsilon$, we obtain $\lim_{r \in \mathcal{R}_+}b_r(f)\gamma_r^{-1}=\lim_{r \in \mathcal{R}_+}\Phi_1f(r, r_-)=0$ by (\ref{fracb_rgamma}).

Suppose that for any $\varepsilon>0$ there exists a finite subset $S_{\varepsilon} \subset \mathcal{R}_+$ such that $|\Phi_1f(r, r_-)|=|b_r(f)\gamma_r^{-1}|<\varepsilon$ for any $r \in \mathcal{R}_+ \setminus S_{\varepsilon}$. (Here, we used (\ref{fracb_rgamma}).) We will show that $\lim_{(x, y) \to (a, a)} \Phi_1f(x, y)=0$ for all $a \in R$. If $a \notin S_\varepsilon$, set $\delta \coloneqq \min \{|r-a| \mid r \in S_\varepsilon \}$. If $r \in \mathcal{R}_+$ satisfies $|r-a|<\delta$, then $r \notin S_\varepsilon$ and $|\Phi_1f(r, r_-)|<\varepsilon$. By putting $B=\{x \in R \mid |x-a| < \delta\}$ and $S=\{x \in K \mid |x| < \varepsilon\}$, Lemma \ref{keylem1} implies that $|\Phi_1f(x, y)|<\varepsilon$ for all distinct $x, y \in S$. If $a \in S_\varepsilon$, set $\delta \coloneqq \min \{|r-a| \mid r \in S_\varepsilon\setminus \{a\} \ \text{or} \ r=a_- \}$. If $r \in \mathcal{R}_+$ satisfies $|r-a|<\delta$, then $r \notin S_\varepsilon$ or $r=a$. Applying Lemma \ref{keylem1} to the same balls $B$ and $S$ as the other case, we obtain the conclusion.
\end{proof}

\subsection{$C^1(R, K)$ is a $K$-Banach space}

To prove Theorem \ref{main} (4) for $n=0$, we use the following lemma, which is called Moore-Osgood's theorem. The proof of Lemma \ref{MOthm} is given by an elementary topology and hence omitted.

\begin{lem}\label{MOthm}
Let $(X, d_X), (Y, d_Y)$ be metric spaces and suppose that $(Y, d_Y)$ is complete. Let $S \subset X$ and $c \in X$ be a limit point of $S$. Assume that a sequence $\{f_n : S \rightarrow Y\}_{n \geq 1}$ is uniformly convergent on $S$ and the limit $\lim_{x \to c}f_n(x)$ exists for each $n \geq 1$. Then the limits $\lim_{x \to c}\lim_{n \to \infty}f_n(x)$ and $\lim_{n \to \infty}\lim_{x \to c}f_n(x)$ exist and satisfy
\begin{align*}
\lim_{x \to c}\lim_{n \to \infty}f_n(x)=\lim_{n \to \infty}\lim_{x \to c}f_n(x).
\end{align*}
\end{lem}

\begin{cor}\label{BanachC^1}
The vector spaces $C^1(R, K)$ and $N^1(R, K)$ are $K$-Banach spaces with respect to the norm $|\cdot|_1$.
\end{cor}

\begin{proof}
We omit the proof for $N^1(R, K)$ because it can be checked by the similar argument to the following proof for $C^1(R, K)$. Let $f=\sum_{r \in \mathcal{R}}b_r(f)\chi_r \in C^1(R, K)$. Since $\Phi_1f : R^2 \rightarrow K$ is continuous, there exists $M>0$ such that $|\Phi_1f|_{\sup}\leq M$. It follows that 
\begin{align*}
|f|_{1}=\sup_{r \in \mathcal{R}}\{|b_r(f)\gamma_r^{-1}|\} \leq \max\{|f(0)|, M\} <\infty
\end{align*}
by (\ref{fracb_rgamma}). It is clear that $|\cdot|_1$ is a norm of $C^1(R, K)$.

We show that $(C^1(R, K), |\cdot|_1)$ is complete. Let $\{f_m\}_{m \geq 1}$ be a Cauchy sequence in $(C^1(R, K), |\cdot|_1)$. That is, for any $\varepsilon >0$ there exists $N \in \mathbb{Z}_{>0}$ such that $|f_l-f_m|<\varepsilon$ for $l, m \geq N$. Then, since 
\begin{align*}
|b_r(f_l)-b_r(f_m)|&\leq |b_r(f_l)-b_r(f_m)||\gamma_r|^{-1}\leq |f_l-f_m|_{1}<\varepsilon
\end{align*}
for any $r \in \mathcal{R}_+$, the sequence $\{b_r(f_m)\}_{m \geq 1}$ is Cauchy in $K$. Put $b_r(f) \coloneqq \lim_{m \to \infty} b_r(f_m)$ and $f\coloneqq \sum_{r \in \mathcal{R}}b_r(f)\chi_r$. It is enough to show that $f \in C^1(R, K)$ and $\lim_{m \to \infty}|f-f_m|_1=0$. Let $a \in R$ and 
$S=\mathcal{R}_+ \setminus \{a\}$. Define $g_m :S \rightarrow K$ to be $g_m(r)=b_r(f_m)\gamma_r^{-1}$. We see that the sequence $\{g_m\}_{m \geq 1}$ is uniformly convergent on $S$. Since $f_m \in C^1(R, K)$ for any $m \geq 1$, the limit $\lim_{\substack{r \to a \\ a \neq r \in \mathcal{R}_+}}g_m(r)$ exits for each $a \in R$ by Theorem \ref{mainforn=0}. Hence, Lemma \ref{MOthm} implies that the limit
\begin{align*}
\lim_{\substack{r \to a \\ a \neq r \in \mathcal{R}_+}}b_r(f)\gamma_r^{-1}=\lim_{m \to \infty}\lim_{\substack{r \to a \\ a \neq r \in \mathcal{R}_+}}g_m(r)
\end{align*}
exists and it follows that $f \in C^1(R, K)$. Finally, since
\begin{align*}
|f-f_m|_{1}=\sup_{r \in \mathcal{R}}\{|g_m(r)-b_r(f)\gamma_r^{-1}|\}<\varepsilon
\end{align*}
for sufficiently large $m \in \mathbb{Z}_{>0}$, we conclude the proof.
\end{proof}

\section{Proof of Theorem \ref{main}}

In the following, if $\operatorname{char}(K)=p>0$, we also assume that $n \leq p-1$. To prove Theorem \ref{main} by induction on $n$, we suppose that the assertions hold for $0, \cdots, n-1$ in this section. Hence, we have
\begin{description}
\item[(IH1)] The $K$-vector spaces $(C(R, K), |\cdot|_{\sup})$ and $(C^j(R, K), |\cdot|_j)$ (where $|\cdot|_j$ is given by (\ref{def of norm})) are $K$-Banach spaces for $1 \leq j \leq n$.
\item[(IH2)] For $0 \leq j \leq n-1$, the set $\{\gamma_r^{j-l}(x-r)^l\chi_r(x) \mid r \in \mathcal{R}, 0 \leq l \leq j\}$ is an orthonormal basis for $C^j(R, K)$.
\item[(IH3)] For any $0 \leq j \leq n-1$ and $f \in C^j(R, K)$, $f$ has the representation $f(x)=\sum_{r \in \mathcal{R}}\sum_{l=0}^j b_r^{j, l}(f)\gamma_r^{j-l}(x-r)^l\chi_r(x)$, where
\begin{align*}
b_r^{j, l}(f)=
\begin{cases}
D_lf(0) & \text{if} \ r=0 \\
\gamma_r \psi_{j-l}D_lf(r, r_-) & \text{if} \ r \neq 0.
\end{cases}
\end{align*}
\item[(IH4)] For any $0 \leq j \leq n-1$ and $f(x)=\sum_{r \in \mathcal{R}}\sum_{l=0}^j b_r^{j, l}(f)\gamma_r^{j-l}(x-r)^l\chi_r(x) \in C^j(R, K)$, then $f \in C^{j+1}(R, K)$ if and only if the limits $\lim_{\substack{r \to a \\ a \neq r \in \mathcal{R}_+}}b_r^{j, l}(f)\gamma_r^{-1}$ exist for all $a \in R$ and $0 \leq l \leq j$ and satisfy
\begin{align*}
\lim_{\substack{r \to a \\ a \neq r \in \mathcal{R}_+}}b_r^{j, l}(f)\gamma_r^{-1}=\binom{j+1}{l}\lim_{\substack{r \to a \\ a \neq r \in \mathcal{R}_+}}b_r^{j, 0}(f)\gamma_r^{-1}.
\end{align*}
\end{description}

\subsection{$C^n$-antiderivation}

To construct an orthonormal basis of $C^n(R, K)$, we introduce a $C^n$-antiderivation and prove some properties. For $x=\sum_{i=0}^\infty c_i\pi^i \in R$, we write
\begin{align*}
x_m=
\begin{cases}
0 & \text{if} \ m=0 \\
\sum_{i=0}^{m-1}c_i\pi^i & \text{if} \ m \geq 1.
\end{cases}
\end{align*}

\begin{dfn}
Let $n \geq 1$. For $f \in C^{n-1}(R, K)$, we define the {\it $C^n$-antiderivation} $P_nf : R \rightarrow K$ to be
\begin{align}\label{defP_n}
P_nf(x)=\sum_{j=0}^{n-1}\sum_{m=0}^\infty \frac{f^{(j)}(x_m)}{(j+1)!}(x_{m+1}-x_m)^{j+1}.
\end{align}
\end{dfn}

It is known that 
\begin{align*}
P_n : C^{n-1}(R, K) \rightarrow C^n(R, K) \ ; \ f \mapsto P_nf
\end{align*}
is $K$-linear and satisfies $(P_nf)'=f$ (\cite[Theorem 81.3]{Sc84}).

\begin{lem}\label{P_nf(r)}
Let $n \geq 1$ and $f \in C^{n-1}(R, K)$. We have
\begin{align*}
P_nf(r)-P_nf(r_-)=\sum_{j=1}^n \frac{(r-r_-)^j}{j!}f^{(j-1)}(r_-)
\end{align*}
for any $r \in \mathcal{R}_+$. (See (\ref{defr_-}) for the definition of $r_-$ and Remark \ref{rmkPhi}(3).)
\end{lem}

\begin{proof}
Let $r=\sum_{i=0}^{l(r)-1} c_i\pi^i \in \mathcal{R}_+$. Considering $r_m=r$ if $m \geq l(r)$ and
\begin{align*}
(r_-)_m=
\begin{cases}
r_m & \text{if} \ m \leq l(r)-1 \\
r_- & \text{if} \ m \geq l(r),
\end{cases}
\end{align*}
we obtain
\begin{align*}
&P_nf(r)-P_nf(r_-) \\
=&\sum_{j=0}^{n-1}\sum_{m=0}^{l(r)-1} \frac{f^{(j)}(r_m)}{(j+1)!}(r_{m+1}-r_m)^{j+1}-\sum_{j=0}^{n-1}\sum_{m=0}^{l(r)-2} \frac{f^{(j)}((r_-)_m)}{(j+1)!}((r_-)_{m+1}-(r_-)_m)^{j+1} \\
=&\sum_{j=0}^{n-1} \frac{f^{(j)}(r_{l(r)-1})}{(j+1)!}(r_{l(r)}-r_{l(r)-1})^{j+1}=\sum_{j=0}^{n-1} \frac{f^{(j)}(r_-)}{(j+1)!}(r-r_-)^{j+1}.
\end{align*}
\end{proof}

\begin{lem}\label{sup_k}
Let $1 \leq k \leq n$. For any $f \in C^k(R, K)$, we have
\begin{align*}
|D_kf|_{\sup} \leq \sup_{r \in \mathcal{R}_+}\{|\psi_{k-1}f(r, r_-)|\} \leq |f|_k.
\end{align*}
(See (\ref{defpsi_jf1}) for the definition of $\psi_{k-1}f$.)
\end{lem}

\begin{proof}
Since $f \in C^k(R, K)$, we have
\begin{align*}
\lim_{(x, y) \to (a, a)} \psi_{k-1}f(x, y)=\lim_{(x, y) \to (a, a)} \Phi_kf(x, y, \cdots, y)=D_kf(a)
\end{align*}
for each $a \in R$. If $a \in R$ satisfies $D_kf(a)=0$, we have $0=|D_kf(a)| \leq |\psi_{k-1}f(r, r_-)|$ for each $r \in \mathcal{R}_+$. For $a \in R$ with $D_kf(a) \neq 0$, there exists $\delta >0$ such that $|\psi_{k-1}f(r, r_-)-D_kf(a)|<|D_kf(a)|$ if $|r-a|<\delta$ and $|r_--a|<\delta$. Then we have
\begin{align*}
|\psi_{k-1}f(r, r_-)|=\max\{|\psi_{k-1}f(r, r_-)-D_kf(a)|, |D_kf(a)|\}=|D_kf(a)|.
\end{align*}
Thus, we see that
\begin{align*}
|D_kf|_{\sup} \leq \sup_{r \in \mathcal{R}_+}\{|\psi_{k-1}f(r, r_-)|\}=\sup_{r \in \mathcal{R}_+}\{|b_r^{k-1, 0}(f)\gamma_r^{-1}|\}\leq |f|_k.
\end{align*}
\end{proof}

\begin{prop}\label{propP_n}
\begin{enumerate}
\item For any $f \in C^{n-1}(R, K)$, we have $|P_nf|_n \leq |(n!)^{-1}f|_{n-1}$. (See (\ref{defP_n}) for the definition of $P_n$.)
\item The $K$-linear map
\begin{align*}
P_n : C^{n-1}(R, K) \rightarrow C^n(R, K) \ ; \ f \mapsto P_nf
\end{align*}
is continuous.
\item For any $0 \leq k \leq n-1$ and $r \in \mathcal{R}$, we have
\begin{align*}
P_n(x-r)^k\chi_r(x)=\frac{1}{k+1}(x-r)^{k+1}\chi_r(x).
\end{align*}
\end{enumerate}
\end{prop}

\begin{proof}
1. Let $r \in \mathcal{R}_+$. We see that
\begin{align*}
b_r^{n-1, j}(P_nf)\gamma_r^{-1}&=\psi_{n-1-j}D_jP_nf(r, r_-) \\
&=\psi_{n-1-j}\binom{j}{1}^{-1}D_{j-1}D_1P_nf(r, r_-) \\
&=j^{-1}\psi_{n-1-j}D_{j-1}f(r, r_-)=j^{-1}b_r^{n-2, j-1}(f)\gamma_r^{-1}
\end{align*}
for $1 \leq j \leq n-1$, where we used (\ref{propertyD_j}) and the induction hypothesis (IH3), and that
\begin{align*}
b_r^{n-1, 0}(P_nf)\gamma_r^{-1}&=\psi_{n-1}P_nf(r, r_-) \\
&=\frac{P_nf(r)-P_nf(r_-)-\sum_{l=1}^{n-1}\gamma_r^lD_lP_nf(r_-)}{\gamma_r^n} \\
&=\frac{1}{n!}f^{(n-1)}(r_-)=\frac{1}{n}D_{n-1}f(r_-)
\end{align*}
by Lemma \ref{P_nf(r)} and (\ref{Dandder}). We also find that $b_0^{n-1, j}(P_nf)=D_jP_nf(0)=j^{-1}D_{j-1}f(0)=j^{-1}b_0^{n-2, j-1}(f)$ for $1 \leq j \leq n-1$, and that $b_0^{n-1, 0}(P_nf)=P_nf(0)=0$. Hence, it follows that
\begin{align*}
|P_nf|_n&=\sup_{r \in \mathcal{R}}\{|b_r^{n-1, 0}(P_nf)\gamma_r^{-1}|, \cdots, |b_r^{n-1, n-1}(P_nf)\gamma_r^{-1}|\} \\
&=\sup_{r \in \mathcal{R}} \left\{ \left|\frac{1}{n}D_{n-1}f(r_-)\right|, |b_r^{n-2, 0}(f)\gamma_r^{-1}|, \cdots, |(n-1)^{-1}b_r^{n-2, n-2}(f)\gamma_r^{-1}|\right\} \\
&\leq \max\left\{\left|\frac{1}{n}D_{n-1}f\right|_{\sup}, \left|\frac{1}{(n-1)!}f\right|_{n-1}\right\} \leq \left|\frac{1}{n!}f\right|_{n-1}.
\end{align*}
Here, we used the induction hypothesis (IH1) Lemma \ref{sup_k} in the last inequality.

2. This follows from the assertion (1).

3. Let $0 \leq k \leq n-1$ and $r \in \mathcal{R}$. If $r \ntriangleleft x$, we see that $P_n(x-r)^k\chi_r(x)=\frac{1}{k+1}(x-r)^{k+1}\chi_r(x)=0$. ($\lhd$ is defined in (\ref{deflhd}).) Suppose that $r \lhd x$. Since 
\begin{align*}
\left((x-r)^k\chi_r(x)\right)^{(j)}=
\begin{cases}
\frac{k!}{(k-j)!}(x-r)^{k-j}\chi_r(x) & \text{if} \ 0 \leq j \leq k \\
0 & \text{if} \ j \geq k+1
\end{cases}
\end{align*}
and $\chi_r(x_m)=0$ for $0 \leq m \leq l(r)-1$, we obtain 
\begin{align*}
P_n(x-r)^k\chi_r(x)&=\sum_{j=0}^{k}\sum_{m=0}^\infty \frac{1}{(j+1)!}\frac{k!}{(k-j)!}(x_m-r)^{k-j}\chi_r(x_m)(x_{m+1}-x_m)^{j+1} \\
&=\frac{1}{k+1}\chi_r(x)\sum_{j=0}^{k}\sum_{m=l(r)}^\infty \binom{k+1}{j+1} (x_m-r)^{k-j}(x_{m+1}-x_m)^{j+1} \\
&=\frac{1}{k+1}\chi_r(x)\sum_{m=l(r)}^\infty \left\{(x_{m+1}-r)^{k+1}-(x_m-r)^{k+1}\right\} \\
&=\frac{1}{k+1}(x-r)^{k+1}\chi_r(x).
\end{align*}
\end{proof}

\begin{prop}
Let $T_n \coloneqq n! P_n \circ \cdots \circ P_1: C(R, K) \rightarrow C^n(R, K)$. Then we have $|T_nf|_n=|f|_{\sup}$ for any $f \in C(R, K)$.
\end{prop}

\begin{proof}
First, we show $|T_n\chi_r|_n=1$ for any $r \in \mathcal{R}$. Let $r_0 \in \mathcal{R}$. Since 
\begin{align*}
P_k(x-r_0)^{k-1}\chi_{r_0}(x)=P_n(x-r_0)^{k-1}\chi_{r_0}(x)=\frac{1}{k}(x-r_0)^k\chi_{r_0}(x)
\end{align*}
for $1 \leq k \leq n$, we find that $T_n\chi_{r_0}(x)=(x-r_0)^n\chi_{r_0}(x)$. To obtain $|T_n\chi_{r_0}|_n$, we compute $b_r^{n-1, j}(T_n\chi_{r_0})$ for $r \in \mathcal{R}$ and $0 \leq j \leq n-1$. Since 
\begin{align*}
D_jT_n\chi_{r_0}&=(j!)^{-1}\left((x-r_0)^n\chi_{r_0}(x)\right)^{(j)} \\
&=\frac{1}{j!}\frac{n!}{(n-j)!}(x-r_0)^{n-j}\chi_{r_0}(x)=\binom{n}{j}(x-r_0)^{n-j}\chi_{r_0}(x),
\end{align*}
we see that $b_0^{n-1, j}(T_n\chi_{r_0})=D_jT_n\chi_{r_0}(0)=0$ for $r=0$. If $r \neq 0$ and $r_0 \lhd r_-$, we obtain
\begin{align*}
b_r^{n-1, j}(T_n\chi_{r_0})&=\gamma_r \psi_{n-1-j}D_jT_n\chi_{r_0}(r, r_-) \\
&=\frac{D_jT_n\chi_{r_0}(r)-D_jT_n\chi_{r_0}(r_-)-\sum_{l=1}^{n-1-j}\gamma_r^lD_lD_jT_n\chi_{r_0}(r_-)}{\gamma_r^{n-1-j}} \\
&=\binom{n}{j}\frac{(r-r_0)^{n-j}-(r_--r_0)^{n-j}-\sum_{l=1}^{n-1-j}\binom{n-j}{l}(r-r_-)^l(r_--r_0)^{n-j-l}}{\gamma_r^{n-1-j}} \\
&=\binom{n}{j}\frac{(r-r_0)^{n-j}-\left\{(r-r_0)^{n-j}-(r-r_-)^{n-j}\right\}}{\gamma_r^{n-1-j}}=\binom{n}{j}\gamma_r.
\end{align*}
By the same computation, it follows that $b_r^{n-1, j}(T_n\chi_{r_0})=0$ if $r_0 \ntriangleleft r_-$. Thus, we find that
\begin{align*}
|T_n\chi_{r_0}|_n&=\sup_{r \in \mathcal{R}_+}\left\{|b_r^{n-1, 0}(T_n\chi_{r_0})\gamma_r^{-1}|, \cdots, |b_r^{n-1, n-1}(T_n\chi_{r_0})\gamma_r^{-1}|\right\} \\
&=\max_{0 \leq j \leq n-1}\left\{ \left|\binom{n}{j}\right| \right\}=1.
\end{align*}

We prove that $|T_nf|_n=|f|_{\sup}$ for $f \in C(R, K)$. Let $f=\sum_{r \in \mathcal{R}}b_r(f)\chi_r \in C(R, K)$. Since $T_nf=\sum_{r \in \mathcal{R}}b_r(f)(x-r)^n\chi_r(x)$, it follows that
\begin{align*}
|T_nf|_n&=\left|\sum_{r \in \mathcal{R}}b_r(f)(x-r)^n\chi_r(x)\right|_n \\
&\leq \sup_{r \in \mathcal{R}} \left\{ \left|b_r(f)(x-r)^n\chi_r(x)\right|_n\right\} \\
&=\sup_{r \in \mathcal{R}}\{|b_r(f)|\}=|f|_{\sup}=|D_nT_nf|_{\sup} \leq |T_nf|_n.
\end{align*}
Here, we used Lemma \ref{sup_k} in the last inequality. We are done.
\end{proof}

\subsection{Proof of Theorem \ref{main} $(1)$ and $(2)$}

\begin{lem}\label{half_corofmain}
Let $n \geq 1$ and $f=\sum_{r \in \mathcal{R}}b_r(f)\chi_r \in C(R, K)$. If $f \in C^n(R, K)$ and $f'=0$, then we have $\lim_{r \in \mathcal{R}_+}b_r(f)\gamma_r^{-n}=0$. (See (\ref{lim_r}) for the definition of this limit.)
\end{lem}

\begin{proof}
By \cite[Theorem 29.12]{Sc84}, the assumption is equivalent to the condition that 
\begin{align*}
\lim_{(x, y) \to (a, a)} \frac{f(x)-f(y)}{(x-y)^n}=0
\end{align*}
for each $a \in R$. The compactness of $\Delta_2$ (see (\ref{defdelta})) implies that for any $\varepsilon>0$ there exists $\delta>0$ such that $|(f(x)-f(y))/(x-y)^n|<\varepsilon$ for all $x, y \in R$ with $|x-y|<\delta$. (Compare (\ref{U_a}).) By setting $S_{\varepsilon} \coloneqq \{r \in \mathcal{R}_+ \mid l(r) \leq 1-\log_q \delta\}$, we see that $|(f(r)-f(r_-))/(r-r_-)^n|<\varepsilon$ for any $r \in \mathcal{R}_+ \setminus S_{\varepsilon}$. This means that
\begin{align*}
\lim_{r \in \mathcal{R}_+}\frac{f(r)-f(r_-)}{(r-r_-)^n}=\lim_{r \in \mathcal{R}_+}\frac{b_r(f)}{\gamma_r^{n}}=0.
\end{align*}
\end{proof}


\begin{proof}[Proof of Theorem \ref{main} (1)]
Since $\{\chi_r \mid r \in \mathcal{R}\}$ is an orthonormal basis for $C(R, K)$ and $T_n$ is norm-preserving, $\{T_n\chi_r=(x-r)^n\chi_r(x) \mid r \in \mathcal{R}\}$ is an orthonormal set in $C^n(R, K)$. Let $c_{j, r} \in K$ for each $r \in \mathcal{R}$ and $0 \leq j \leq n-1$ and put $f=\sum_{r \in \mathcal{R}}\sum_{j=0}^{n-1} c_{j, r} \gamma_r^{n-j}(x-r)^j\chi_r \in C^n(R, K) \subset C^{n-1}(R, K)$. Then we see that $|\gamma_r^{n-j}(x-r)^j\chi_r|_n=1$ for each $r \in \mathcal{R}$ and $0 \leq j \leq n-1$ and that
\begin{align*}
|f|_n&=\sup_{r \in \mathcal{R}}\{|c_{0, r}\gamma_r\gamma_r^{-1}|, \cdots, |c_{n-1, r}\gamma_r\gamma_r^{-1}|\} \\
&=\sup_{r \in \mathcal{R}}\{|c_{0, r}|, \cdots, |c_{n-1, r}|\}.
\end{align*}
Hence, $\{\gamma_r^n\chi_r, \gamma_r^{n-1}(x-r)\chi_r, \cdots, \gamma_r(x-r)^{n-1}\chi_r \mid r \in \mathcal{R}\}$ is orthonormal in $C^n(R, K)$.

We prove $\{\gamma_r^n\chi_r, \gamma_r^{n-1}(x-r)\chi_r, \cdots, \gamma_r(x-r)^{n-1}\chi_r, (x-r)^n\chi_r \mid r \in \mathcal{R}\}$ is orthonormal in $C^{n}(R, K)$. It suffices to show that 
\begin{align*}
|f|_n=\sup_{r \in \mathcal{R}}\{|c_{0, r}|, \cdots, |c_{n-1, r}|, |c_{n, r}|\}
\end{align*}
for $f=\sum_{r \in \mathcal{R}}\sum_{j=0}^{n} c_{j, r} \gamma_r^{n-j}(x-r)^j\chi_r \in C^n(R, K)$. Set $N_n^n(R, K)=\{f \in C^n(R, K) \mid f^{(n)}=0\}$. Since, for any $f \in N_n^n(R, K)$, $g \in C(R, K)$ and $\lambda \in K$, 
\begin{align*}
|T_ng-\lambda f|_n & \geq |D_n(T_ng-\lambda f)|_{\sup} \\
&=\left|\frac{1}{n!}(T_ng-\lambda f)^{(n)}\right|_{\sup}=|g|_{\sup}=|T_ng|_n,
\end{align*}
we have $N_n^n(R, K) \perp \operatorname{Im} T_n$ in $C^n(R, K)$. Since $\sum_{r \in \mathcal{R}}\sum_{j=0}^{n-1} c_{j, r} \gamma_r^{n-j}(x-r)^j\chi_r \in N_n^n(R, K)$ and $\sum_{r \in \mathcal{R}}c_{n,r}(x-r)^n\chi_r \in \operatorname{Im} T_n$, we obtain
\begin{align*}
|f|_n&\geq \max\left\{\left|\sum_{r \in \mathcal{R}}\sum_{j=0}^{n-1} c_{j, r} \gamma_r^{n-j}(x-r)^j\chi_r\right|_n, \left|\sum_{r \in \mathcal{R}}c_{n,r}(x-r)^n\chi_r\right|_n\right\} \\
&=\max\left\{\sup_{r \in \mathcal{R}}\{|c_{0, r}|, \cdots, |c_{n-1, r}|\}, \sup_{r \in \mathcal{R}}\{|c_{n, r}|\}\right\} \\
&=\sup_{r \in \mathcal{R}}\{|c_{0, r}|, \cdots, |c_{n-1, r}|, |c_{n, r}|\}.
\end{align*}
On the other hand, since
\begin{align*}
|f|_n&\leq \max_{0 \leq j \leq n}\left\{\left|\sum_{r \in \mathcal{R}} c_{j, r} \gamma_r^{n-j}(x-r)^j\chi_r\right|_n\right\} \\
&=\max_{0 \leq j \leq n}\left\{\sup_{r \in \mathcal{R}}\{|c_{j, r}|\}\right\} \\
&=\sup_{r \in \mathcal{R}}\{|c_{0, r}|, \cdots, |c_{n-1, r}|, |c_{n, r}|\},
\end{align*}
it follows that $\{\gamma_r^n\chi_r, \gamma_r^{n-1}(x-r)\chi_r, \cdots, \gamma_r(x-r)^{n-1}\chi_r, (x-r)^n\chi_r \mid r \in \mathcal{R}\}$ is an orthonormal set in $C^{n}(R, K)$.

Finally, we check that $\{\gamma_r^n\chi_r, \gamma_r^{n-1}(x-r)\chi_r, \cdots, \gamma_r(x-r)^{n-1}\chi_r, (x-r)^n\chi_r \mid r \in \mathcal{R}\}$ is a basis for $C^{n}(R, K)$. For a given $f \in C^n(R, K)$, since $f' \in C^{n-1}(R, K)$, $f'$ has the representation
\begin{align*}
f'=\sum_{r \in \mathcal{R}}\sum_{j=0}^{n-1} b_r^{n-1, j}(f')\gamma_r^{n-1-j}(x-r)^j\chi_r
\end{align*}
in $C^{n-1}(R, K)$. Note that
\begin{align*}
b_r^{n-1, j}(f')=
\begin{cases}
D_jf'(0) & \text{if} \ r=0 \\
\gamma_r \psi_{n-1-j}D_jf'(r, r_-) & \text{if} \ r \neq 0
\end{cases}
\end{align*}
for $0 \leq j \leq n-1$, by the induction hypothesis (IH3). It follows from this representation and Proposition \ref{propP_n} that 
\begin{align*}
P_nf'=\sum_{r \in \mathcal{R}}\sum_{j=0}^{n-1} \frac{1}{j+1}b_r^{n-1, j}(f')\gamma_r^{n-1-j}(x-r)^{j+1}\chi_r \in C^n(R, K).
\end{align*}
Then, by putting $g \coloneqq f-P_nf'$, we see that $g \in C^n(R, K)$ and $g'=0$. Thus, we find that $\lim_{r \in \mathcal{R}_+}b_r(g)\gamma_r^{-n}=0$ by Lemma \ref{half_corofmain} and the infinite sum 
\begin{align*}
g=\sum_{r \in \mathcal{R}} \frac{b_r(g)}{\gamma_r^n}\gamma_r^n\chi_r
\end{align*}
converges in $C^n(R, K)$ (with respect to the norm $|\cdot|_n$). Hence, we obtain 
\begin{align*}
f&=g+P_nf' \\
&=\sum_{r \in \mathcal{R}}\left\{\frac{b_r(g)}{\gamma_r^n}\gamma_r^n\chi_r+\sum_{j=1}^{n} \frac{1}{j}b_r^{n-1, j-1}(f')\gamma_r^{n-j}(x-r)^j\chi_r\right\}.
\end{align*}
\end{proof}

\begin{proof}[Proof of Theorem \ref{main} (2)]
Let $f \in C^n(R, K)$. We keep the notations in the proof of (1) and compute $b_r^{n, j}(f)$ for each $0 \leq j \leq n$ by using the above proof of Theorem \ref{main} (1).

First, we compute $b_r^{n, 0}(f)$. If $r=0$, we have
\begin{align*}
b_0^{n, 0}(f)=b_0(g)=f(0)-P_nf'(0)=f(0)=D_0f(0).
\end{align*}
For $r \neq 0$, since 
\begin{align*}
b_r(g)&=f(r)-f(r_-)-P_nf'(r)+P_nf'(r_-) \\
&=f(r)-f(r_-)-\sum_{j=1}^n\frac{(r-r_-)^j}{j!}f^{(j)}(r_-) \\
&=f(r)-f(r_-)-\sum_{j=1}^n\gamma_r^jD_jf(r_-),
\end{align*}
we obtain
\begin{align*}
b_r^{n, 0}(f)=\frac{f(r)-f(r_-)-\sum_{j=1}^n\gamma_r^jD_jf(r_-)}{\gamma_r^n}=\gamma_r \psi_{n}f(r, r_-).
\end{align*}

Finally, it follows that
\begin{align*}
b_r^{n, j}(f)& =j^{-1}b_r^{n-1, j-1}(f') \\
&=
\begin{cases}
j^{-1}(D_{j-1}D_1f)(0)=D_jf(0) & \text{if} \ r=0 \\
j^{-1}\gamma_r\psi_{n-j}D_{j-1}D_1f(r, r_-)=\gamma_r\psi_{n-j}D_jf(r, r_-) & \text{if} \ r \neq 0
\end{cases}
\end{align*}
for each $1 \leq j \leq n$.
\end{proof}

\subsection{Generalizations of Lemma \ref{keylem1}}

We prepare several theorems to prove Theorem \ref{main} $(3)$ and $(4)$. The following theorem is a generalization of Lemma \ref{keylem1}.

\begin{thm}\label{keylem2}
Let $n \geq 0$, $f \in C^n(R, K)$, $a \in R$, $c \in K$, and $\delta, \varepsilon >0$. Suppose that 
\begin{align*}
\left|\psi_{n-j}D_jf(r, r_-)-\binom{n+1}{j}c\right|<\varepsilon
\end{align*}
for any $0 \leq j \leq n$ and $r \in \mathcal{R}_+$ with $|r-a|<\delta$ and $|r_--a|<\delta$. Then we have 
\begin{align*}
\left|\psi_nf(x, y)-c\right|<\varepsilon
\end{align*}
for any distinct $x, y \in R$ with $|x-a|<\delta$ and $|y-a|<\delta$.
\end{thm}

Note that Theorem \ref{keylem2} for $n=0$ coincides with Lemma \ref{keylem1}. To prove Theorem \ref{keylem2}, we prepare two lemmas.

\begin{lem}\label{lemlem1}
Let $n \geq 0$ and $f \in C^n(R, K)$. For pairwise distinct elements $x, y, z \in R$, we have
\begin{align*}
\psi_nf(x, y)=\left(\frac{x-z}{x-y}\right)^{n+1}\psi_nf(x, z)-\sum_{l=0}^n\left(\frac{y-z}{x-y}\right)^{n+1-l}\psi_{n-l}D_lf(y, z).
\end{align*}
\end{lem}

\begin{proof}
We check the assertion by induction on $n$. We already proved for $n=0$ in (\ref{lemlem1.1}). Let $n \geq 0$ and suppose the assertion holds for $n$. Then we see that
\begin{align*}
&\left(\frac{x-z}{x-y}\right)^{n+2}\psi_{n+1}f(x, z)-\sum_{l=0}^{n+1}\left(\frac{y-z}{x-y}\right)^{n+2-l}\psi_{n+1-l}D_lf(y, z) \\
=&\left(\frac{x-z}{x-y}\right)^{n+2}\frac{\psi_{n}f(x, z)-D_{n+1}f(z)}{x-z}-\left(\frac{y-z}{x-y}\right)\frac{D_{n+1}f(y)-D_{n+1}f(z)}{y-z} \\
&-\sum_{l=0}^{n}\left(\frac{y-z}{x-y}\right)^{n+2-l}\frac{\psi_{n-l}D_lf(y, z)-D_{n+1-l}D_lf(z)}{y-z} \\
=&\frac{1}{x-y}\left\{\left(\frac{x-z}{x-y}\right)^{n+1}\psi_nf(x, z)-\sum_{l=0}^n\left(\frac{y-z}{x-y}\right)^{n+1-l}\psi_{n-l}D_lf(y, z)\right\}-\frac{D_{n+1}f(y)}{x-y} \\
&-\frac{D_{n+1}f(z)}{(x-y)^{n+2}}\left\{(x-z)^{n+1}-\sum_{l=0}^n\binom{n+1}{l}(x-y)^l(y-z)^{n+1-l}-(x-y)^{n+1}\right\} \\
=&\frac{\psi_nf(x, y)-D_{n+1}f(y)}{x-y}-\frac{D_{n+1}f(z)}{(x-y)^{n+2}}\left\{(x-z)^{n+1}-(x-z)^{n+1}\right\} \\
=&\psi_{n+1}f(x, y).
\end{align*}
Here, we used (\ref{defpsi_jf2}) in the first and fourth equalities, the induction hypothesis in the second and the third equalities, and (\ref{propertyD_j}) in the second equality. Hence, the assertion also holds for $n+1$.
\end{proof}

\begin{lem}\label{lemlem2}
Let $n \geq 0$, $m \geq 2$ and $f \in C^n(R, K)$. For pairwise distinct elements $t_1, \cdots, t_m \in R$, $2 \leq j \leq m$ and $1 \leq l \leq n$, put 
\begin{align*}
\lambda_j^{(n)}=\left(\frac{t_j-t_{j-1}}{t_m-t_1}\right)^{n+1}, \ \mu_{l, j}^{(n)}=\frac{(t_j-t_{j-1})^l(t_{j-1}-t_1)^{n+1-l}}{(t_m-t_1)^{n+1}}.
\end{align*}
Then we have
\begin{align}\label{lemlem2former}
\psi_nf(t_m, t_1)=\sum_{j=2}^m \lambda_j^{(n)}\psi_nf(t_j, t_{j-1})+\sum_{l=1}^n\sum_{j=3}^m \mu_{l, j}^{(n)}\psi_{n-l}D_lf(t_{j-1}, t_1)
\end{align}
and
\begin{align}\label{lemlem2later}
\sum_{j=2}^m \lambda_j^{(n)}+\sum_{l=1}^n\sum_{j=3}^m \binom{n+1}{l}\mu_{l, j}^{(n)}=1.
\end{align}
Here, the empty sum is understood to be 0.
\end{lem}

\begin{proof}
We assume that $m \geq 3$ since it is clear for $m=2$. For (\ref{lemlem2later}), we see that
\begin{align*}
&\sum_{j=2}^m\left(\frac{t_j-t_{j-1}}{t_m-t_1}\right)^{n+1}+\sum_{l=1}^n\sum_{j=3}^m\binom{n+1}{l}\frac{(t_j-t_{j-1})^l(t_{j-1}-t_1)^{n+1-l}}{(t_m-t_1)^{n+1}} \\
=&\frac{1}{(t_m-t_1)^{n+1}}\sum_{j=3}^m\sum_{l=1}^{n+1}\binom{n+1}{l}(t_j-t_{j-1})^l(t_{j-1}-t_1)^{n+1-l}+\left(\frac{t_2-t_{1}}{t_m-t_1}\right)^{n+1} \\
=&\frac{1}{(t_m-t_1)^{n+1}}\sum_{j=3}^m\left\{(t_j-t_1)^{n+1}-(t_{j-1}-t_1)^{n+1}\right\}+\left(\frac{t_2-t_{1}}{t_m-t_1}\right)^{n+1} \\
=&\frac{1}{(t_m-t_1)^{n+1}}\left\{(t_m-t_1)^{n+1}-(t_2-t_1)^{n+1}\right\}+\left(\frac{t_2-t_{1}}{t_m-t_1}\right)^{n+1}=1.
\end{align*}
We prove (\ref{lemlem2former}) by induction on $n$. We already shown for $n=0$ in (\ref{lemlem2.1}). Let $n \geq 0$ and suppose that (\ref{lemlem2former}) holds for $n$. Then we have 
\begin{align*}
&\sum_{j=2}^m \lambda_j^{(n+1)}\psi_{n+1}f(t_j, t_{j-1})+\sum_{l=1}^{n+1}\sum_{j=3}^m \mu_{l, j}^{(n+1)}\psi_{n+1-l}D_lf(t_{j-1}, t_1) \\
=&\sum_{j=2}^m \left(\frac{t_j-t_{j-1}}{t_m-t_1}\right)^{n+2}\frac{\psi_nf(t_j, t_{j-1})-D_{n+1}f(t_{j-1})}{t_j-t_{j-1}} \\
&+\sum_{l=1}^{n}\sum_{j=3}^m \frac{(t_j-t_{j-1})^l(t_{j-1}-t_1)^{n+2-l}}{(t_m-t_1)^{n+2}}\cdot\frac{\psi_{n-l}D_lf(t_{j-1}, t_1)-D_{n+1-l}D_lf(t_1)}{t_{j-1}-t_1}  \\
&+\sum_{j=3}^m \frac{(t_j-t_{j-1})^{n+1}(t_{j-1}-t_1)}{(t_m-t_1)^{n+2}}\cdot\frac{D_{n+1}f(t_{j-1})-D_{n+1}f(t_{1})}{t_{j-1}-t_{1}} \\
=&\frac{1}{t_m-t_1}\left\{\sum_{j=2}^m \lambda_j^{(n)}\psi_nf(t_j, t_{j-1})+\sum_{l=1}^n\sum_{j=3}^m \mu_{l, j}^{(n)}\psi_{n-l}D_lf(t_{j-1}, t_1)\right\} \\
&-\sum_{l=1}^{n}\sum_{j=3}^m \binom{n+1}{l}\frac{(t_j-t_{j-1})^l(t_{j-1}-t_1)^{n+1-l}}{(t_m-t_1)^{n+2}}D_{n+1}f(t_1) \\
&-\sum_{j=3}^m \frac{(t_j-t_{j-1})^{n+1}}{(t_m-t_1)^{n+2}}D_{n+1}f(t_1)-\frac{(t_2-t_{1})^{n+1}}{(t_m-t_1)^{n+2}}D_{n+1}f(t_1) \\
=&\frac{\psi_nf(x, y)}{t_m-t_1}-\frac{1}{(t_m-t_1)^{n+2}}\sum_{j=3}^mD_{n+1}f(t_1)\sum_{l=1}^{n+1}\binom{n+1}{l}(t_j-t_{j-1})^l(t_{j-1}-t_1)^{n+1-l} \\
&-\frac{(t_2-t_{1})^{n+1}}{(t_m-t_1)^{n+2}}D_{n+1}f(t_1) \\
=&\frac{\psi_nf(x, y)}{t_m-t_1}-\frac{1}{(t_m-t_1)^{n+2}}\sum_{j=3}^m\left\{(t_j-t_1)^{n+1}-(t_{j-1}-t_1)^{n+1}\right\}D_{n+1}f(t_1) \\
&-\frac{(t_2-t_{1})^{n+1}}{(t_m-t_1)^{n+2}}D_{n+1}f(t_1) \\
=&\frac{\psi_nf(x, y)}{t_m-t_1}-\frac{1}{(t_m-t_1)^{n+2}}\left\{(t_m-t_1)^{n+1}-(t_2-t_1)^{n+1}\right\}D_{n+1}f(t_1) \\
&-\frac{(t_2-t_{1})^{n+1}}{(t_m-t_1)^{n+2}}D_{n+1}f(t_1) \\
=&\frac{\psi_nf(t_m, t_1)-D_{n+1}f(t_1)}{t_m-t_1}=\psi_{n+1}f(t_m, t_1).
\end{align*}
\end{proof}

We prove Theorem \ref{keylem2} in a similar way to the proof of Lemma \ref{keylem1}.

\begin{proof}[Proof of Theorem \ref{keylem2}]
We prove the assertion by induction on $n$. For $n=0$, we already proved Lemma \ref{keylem1}. Let $n >0$ and suppose that the assertions hold for $0, 1, \cdots, n-1$.

For the same reason as the proof of Lemma \ref{keylem1}, we may assume that $x, y \in \mathcal{R}_+$, $|x-a|<\delta$ and $|y-a|<\delta$. Set $z$ to be (\ref{defz}) (i.e. $z$ is the common initial part in the $\pi$-adic expansions of $x$ and $y$). By the definition of $z$, we see that $z \lhd x, z \lhd y$, $|z-a|<\delta$ and $|x-y|=\max\{|z-x|, |z-y|\}$. Since 
\begin{align*}
\psi_nf(x, y)=\left(\frac{x-z}{x-y}\right)^{n+1}\psi_nf(x, z)-\sum_{l=0}^n\left(\frac{y-z}{x-y}\right)^{n+1-l}\psi_{n-l}D_lf(y, z)
\end{align*}
by Lemma \ref{lemlem1} and
\begin{align*}
\left(\frac{x-z}{x-y}\right)^{n+1}-\sum_{l=0}^n\binom{n+1}{l}\left(\frac{y-z}{x-y}\right)^{n+1-l}=1,
\end{align*}
it follows that
\begin{align*}
&|\psi_nf(x, y)-c| \\
=&\left|\left(\frac{x-z}{x-y}\right)^{n+1}\psi_nf(x, z)-\sum_{l=0}^n\left(\frac{y-z}{x-y}\right)^{n+1-l}\psi_{n-l}D_lf(y, z) \right. \\
&\left.-\left\{\left(\frac{x-z}{x-y}\right)^{n+1}-\sum_{l=0}^n\binom{n+1}{l}\left(\frac{y-z}{x-y}\right)^{n+1-l}\right\}c\right| \\
\leq &\max_{0 \leq l \leq n}\left\{\left|\frac{x-z}{x-y}\right|^{n+1}|\psi_nf(x, z)-c|, \ \left|\frac{y-z}{x-y}\right|^{n+1-l}\left|\psi_{n-l}D_lf(y, z)-\binom{n+1}{l}c\right|\right\}.
\end{align*}
Let $1 \leq l \leq n$. For any $0 \leq i \leq n-l$ and any $r \in \mathcal{R}_+$ with $|r-a|<\delta$ and $|r_--a|<\delta$, we have
\begin{align*}
&\left|\psi_{n-l-i}D_iD_lf(r,r_-)-\binom{n+1-l}{i}\binom{n+1}{l}c\right| \\
=&\left|\binom{i+l}{l}\psi_{n-i-l}D_{i+l}f(r,r_-)-\binom{n+1-l}{i}\binom{n+1}{l}c\right| \\
=&\left|\binom{i+l}{l}\left\{\psi_{n-i-l}D_{i+l}f(r,r_-)-\binom{n+1}{i+l}c\right\}\right|<\varepsilon
\end{align*}
by the assumption. Hence, it follows from the induction hypothesis that 
\begin{align}\label{cons_of_hyp}
\left|\psi_{n-l}D_lf(y, z)-\binom{n+1}{l}c\right|<\varepsilon
\end{align}
for any distinct $y, z \in R$ with $|y-a|<\delta$ and $|z-a|<\delta$ and it suffices to show that $|\psi_nf(x, z)-c|<\varepsilon$ and $|\psi_nf(y, z)-c|<\varepsilon$. Thus, we may assume that $y \lhd x$ by replacing $z$ with $y$. Then there exists the unique sequence $t_1=y \lhd t_2 \lhd \cdots \lhd t_n=x$ in $\mathcal{R}$ such that $(t_j)_-=t_{j-1}$ for each $2 \leq j \leq n$ and $|t_j-a|<\delta$ for each $1 \leq j \leq n$. Put
\begin{align*}
\lambda_j^{(n)}=\left(\frac{t_j-t_{j-1}}{x-y}\right)^{n+1}, \ \mu_{l, j}^{(n)}=\frac{(t_j-t_{j-1})^l(t_{j-1}-y)^{n+1-l}}{(x-y)^{n+1}}
\end{align*}
for each $2 \leq j \leq m$ and $1 \leq l \leq n$. Then we see that $|\lambda_j^{(n)}|\leq 1$, $|\mu_{l, j}^{(n)}|\leq 1$, 
\begin{align*}
\psi_nf(x, y)=\sum_{j=2}^m \lambda_j^{(n)}\psi_nf(t_j, t_{j-1})+\sum_{l=1}^n\sum_{j=3}^m \mu_{l, j}^{(n)}\psi_{n-l}D_lf(t_{j-1}, y)
\end{align*}
and
\begin{align*}
\sum_{j=2}^m \lambda_j^{(n)}+\sum_{l=1}^n\sum_{j=3}^m \binom{n+1}{l}\mu_{l, j}^{(n)}=1
\end{align*}
by Lemma \ref{lemlem2}. Hence, we obtain 
\begin{align*}
&|\psi_nf(x, y)-c| \\
=&\left|\sum_{j=2}^m \lambda_j^{(n)}\psi_nf(t_j, t_{j-1})+\sum_{l=1}^n\sum_{j=3}^m \mu_{l, j}^{(n)}\psi_{n-l}D_lf(t_{j-1}, y)\right. \\
&\left.-\left\{\sum_{j=2}^m \lambda_j^{(n)}+\sum_{l=1}^n\sum_{j=3}^m \binom{n+1}{l}\mu_{l, j}^{(n)}\right\}c\right| \\
\leq&\max\left\{ \max_{2 \leq j \leq m}\left\{|\lambda_j^{(n)}||\psi_nf(t_j, t_{j-1})-c|\right\}, \ \max_{\substack{1 \leq l \leq n \\ 3 \leq j \leq m}}\left\{|\mu_{l, j}^{(n)}|\left|\psi_{n-l}D_lf(t_{j-1}, y)-\binom{n+1}{l}c\right|\right\}\right\} \\
<&\varepsilon
\end{align*}
by using (\ref{cons_of_hyp}) and the induction hypothesis.
\end{proof}

\begin{dfn}
Let $f \in C^n(R, K)$ and $1 \leq j \leq n+1$. We define the continuous function $\psi_{n, j}f : \bigtriangledown^2R \rightarrow K$ to be
\begin{align*}
\psi_{n, j}f(x, y) \coloneqq \Phi_{n+1}f(\underbrace{x, \cdots, x}_{j}, \underbrace{y, \cdots, y}_{n+2-j}).
\end{align*}
\end{dfn}

Note that $\psi_{n, 1}f(x, y)=\psi_nf(x, y)$. It is known that the following lemmas hold.

\begin{lem}[{\cite[Lemma 78.3]{Sc84}}]\label{lem78.3}
Let $n \geq 1$ and $f \in C^n(R, K)$. For any $1 \leq j \leq n$, we have
\begin{align*}
\psi_{n-j}D_jf(x, y)=\sum_{i=1}^{j+1}\binom{n+1-i}{n-j}\psi_{n, i}f(x, y).
\end{align*}
\end{lem}

\begin{lem}[{\cite[Lemma 81.2]{Sc84}}]\label{lem81.2}
Let $n \geq 0$, $f \in C^n(R, K)$, $a \in R$, $c \in K$, and $\delta, \varepsilon >0$. Suppose that 
\begin{align*}
\left|\psi_{n, j}f(x, y)-c\right|<\varepsilon
\end{align*}
for any $1 \leq j \leq n+1$ and any distinct $x, y \in R$ with $|x-a|<\delta$ and $|y-a|<\delta$. Then we have
\begin{align*}
|\Phi_{n+1}f(x_1, \cdots, x_{n+2})-c|<\varepsilon
\end{align*}
for any $(x_1, \cdots, x_{n+2}) \in R^{n+2} \setminus \Delta_{n+2}$ where $|x_i-a|<\delta$ for each $1 \leq i \leq n+2$. (See (\ref{defdelta}) for the definition of $\Delta_{n+2}$.)
\end{lem}

We show the following theorem.

\begin{thm}\label{keylem3}
Let $n \geq 0$, $f \in C^n(R, K)$, $a \in R$, $c \in K$, and $\delta, \varepsilon >0$. Suppose that 
\begin{align*}
\left|\psi_{n-j}D_jf(r, r_-)-\binom{n+1}{j}c\right|<\varepsilon
\end{align*}
for any $0 \leq j \leq n$ and $r \in \mathcal{R}_+$ with $|r-a|<\delta$ and $|r_--a|<\delta$. Then we have 
\begin{align*}
\left|\psi_{n, j}f(x, y)-c\right|<\varepsilon
\end{align*}
for any $1 \leq j \leq n+1$ and any distinct $x, y \in R$ with $|x-a|<\delta$ and $|y-a|<\delta$.
\end{thm}

\begin{proof}
We prove the assertion by induction on $j$. We already proved the assertion for $j=1$ in Theorem \ref{keylem2}. Let $1<j \leq n+1$ and suppose that the assertions hold for $1, \cdots, j-1$. Then we have
\begin{align*}
&\left|\psi_{n, j}f(x, y)-c\right| \\
=&\left|\psi_{n+1-j}D_{j-1}f(x, y)-\sum_{i=1}^{j-1}\binom{n+1-i}{n+1-j}\psi_{n, i}f(x, y)-\left\{\binom{n+1}{j-1}-\sum_{i=1}^{j-1}\binom{n+1-i}{n+1-j}\right\}c\right| \\
\leq & \max_{1 \leq i \leq j-1}\left\{\left|\psi_{n+1-j}D_{j-1}f(x, y)-\binom{n+1}{j-1}c\right|, \left|\binom{n+1-i}{n+1-j}\right|\left|\psi_{n, i}f(x, y)-c\right|\right\}
\end{align*}
for any distinct $x, y \in R$ with $|x-a|<\delta$ and $|y-a|<\delta$. Here, we used Lemma \ref{lem78.3} and 
\begin{align*}
\binom{n+1}{j-1}-\sum_{i=1}^{j-1}\binom{n+1-i}{n+1-j}=1
\end{align*}
in the first equality. We obtain $|\psi_{n, j}f(x, y)-c|<\varepsilon$ by using (\ref{cons_of_hyp}) and the induction hypothesis.
\end{proof}

It is clear that the following corollary follows form Lemma \ref{lem81.2} and Theorem \ref{keylem3}.

\begin{cor}\label{keycor}
Let $n \geq 0$, $f \in C^n(R, K)$, $a \in R$, $c \in K$, and $\delta, \varepsilon >0$. Suppose that 
\begin{align*}
\left|\psi_{n-j}D_jf(r, r_-)-\binom{n+1}{j}c\right|<\varepsilon
\end{align*}
for any $0 \leq j \leq n$ and $r \in \mathcal{R}_+$ with $|r-a|<\delta$ and $|r_--a|<\delta$. Then we have
\begin{align*}
|\Phi_{n+1}f(x_1, \cdots, x_{n+2})-c|<\varepsilon
\end{align*}
for any $(x_1, \cdots, x_{n+2}) \in R^{n+2} \setminus \Delta_{n+2}$ where $|x_i-a|<\delta$ for each $1 \leq i \leq n+2$.
\end{cor}

\subsection{Proof of Theorem \ref{main} $(3)$ and $(4)$}

We show Theorem \ref{main} $(3)$ and $(4)$.

\begin{proof}[Proof of Theorem \ref{main} (3)]
Suppose that $f \in C^{n+1}(R, K)$. Since $D_jf \in C^{n+1-j}(R, K)$ for each $0 \leq j \leq n$, we have
\begin{align*}
D_jf(x)=D_jf(y)+\sum_{l=1}^{n-j}(x-y)^lD_lD_jf(y)+(x-y)^{n+1-j}\Phi_{n+1-j}D_jf(x, y, \cdots, y)
\end{align*}
for any $x, y \in R$ by \cite[Theorem 29.3]{Sc84}. Hence, by (\ref{defpsi_jf1}) and (\ref{propertyD_j}), we obtain 
\begin{align*}
\lim_{(x, y) \to (a, a)} \psi_{n-j}D_jf(x, y)&=\lim_{(x, y) \to (a, a)}\Phi_{n+1-j}D_jf(x, y,\cdots, y) \\
&=D_{n+1-j}D_jf(a)=\binom{n+1}{j}D_{n+1}f(a)
\end{align*}
for any $a \in R$. We see that for any $\varepsilon>0$ there exists $\delta_0>0$ such that 
\begin{align*}
\left|\psi_{n-j}D_jf(r, r_-)-\binom{n+1}{j}D_{n+1}f(a)\right|<\varepsilon
\end{align*}
for any $r \in \mathcal{R}_+$ with $0<|r-a|<\delta_0$. (Compare (\ref{conclusion}).) Thus, the limits
\begin{align*}
\lim_{\substack{r \to a \\ a \neq r \in \mathcal{R}_+}}b_r^{n, j}(f)\gamma_r^{-1}=\lim_{\substack{r \to a \\ a \neq r \in \mathcal{R}_+}}\psi_{n-j}D_jf(x, y)=\binom{n+1}{j}D_{n+1}f(a)
\end{align*}
exist. Since 
\begin{align*}
\lim_{\substack{r \to a \\ a \neq r \in \mathcal{R}_+}}b_r^{n, 0}(f)\gamma_r^{-1}=D_{n+1}f(a),
\end{align*}
we find that
\begin{align*}
\lim_{\substack{r \to a \\ a \neq r \in \mathcal{R}_+}}b_r^{n, j}(f)\gamma_r^{-1}=\binom{n+1}{j}\lim_{\substack{r \to a \\ a \neq r \in \mathcal{R}_+}}b_r^{n, 0}(f)\gamma_r^{-1}
\end{align*}
for each $0 \leq j \leq n$.

Conversely, we suppose that the limits $\lim_{\substack{r \to a \\ a \neq r \in \mathcal{R}_+}}b_r^{n, j}(f)\gamma_r^{-1}$ exist for any $a \in R$ and $0 \leq j \leq n$ and 
\begin{align*}
\lim_{\substack{r \to a \\ a \neq r \in \mathcal{R}_+}}b_r^{n, j}(f)\gamma_r^{-1}=\binom{n+1}{j}\lim_{\substack{r \to a \\ a \neq r \in \mathcal{R}_+}}b_r^{n, 0}(f)\gamma_r^{-1}
\end{align*}
holds for each $0 \leq j \leq n$. Put $g(a) \coloneqq \lim_{\substack{r \to a \\ a \neq r \in \mathcal{R}_+}}b_r^{n, 0}(f)\gamma_r^{-1}$. Then, we find that for any $\varepsilon>0$ there exists $\delta_0>0$ such that 
\begin{align*}
\left|\psi_{n-j}D_jf(r, r_-)-\binom{n+1}{j}g(a)\right|<\varepsilon
\end{align*}
for any $0 \leq j \leq n$ and any $r \in \mathcal{R}_+$ with $|x-a|<\delta_0$ and $|y-a|<\delta_0$. (Compare (\ref{conclusion2}) and (\ref{conclusion3}).) Hence, Corollary \ref{keycor} implies that 
\begin{align*}
|\Phi_{n+1}f(x_1, \cdots, x_{n+2})-g(a)|<\varepsilon
\end{align*}
for any $(x_1, \cdots, x_{n+2}) \in R^{n+2} \setminus \Delta_{n+2}$ with $|x_i-a|<\delta_0$ for each $1 \leq i \leq n+2$. (We defined $\Delta_{n+2}$  in (\ref{defdelta}).) In other words, it follows that 
\begin{align*}
\lim_{\substack{(x_1, \cdots, x_{n+1}, x_{n+2}) \to (a, \cdots, a) \\ (x_1, \cdots, x_{n+1}, x_{n+2}) \in \bigtriangledown^{n+2}R}} \Phi_{n+2}f(x_1, \cdots, x_{n+1}, x_{n+2})=g(a)
\end{align*}
and $f \in C^{n+1}(R, K)$.
\end{proof}

\begin{proof}[Proof of Theorem \ref{main} (4)]
We give a proof by a similar argument to the proof of Corollary \ref{BanachC^1}. Let $\{f_m\}_{m \geq 1}$ be a Cauchy sequence in $(C^{n+1}(R, K), |\cdot|_{n+1})$ and put $b_r^{n, j}(f)\coloneqq \lim_{m \to \infty}b_r^{n, j}(f_m)$ and $f \coloneqq \sum_{r \in \mathcal{R}}\sum_{j=0}^nb_r^{n, j}(f)\gamma_r^{n-j}(x-r)^j\chi_r$ for $0 \leq j \leq n$. We show that $f \in C^{n+1}(R, K)$ and $\lim_{m \to \infty}|f-f_m|_{n+1}=0$. Let $a \in R$ and $S=\mathcal{R}_+ \setminus \{a\}$. Define $g_m^{n, j} :S \rightarrow K$ to be $g_m^{n, j}(r)=b_r^{n, j}(f_m)\gamma_r^{-1}$ for $m \geq 1$ and $0 \leq j \leq n$. By the same reason in the proof of Corollary \ref{BanachC^1}, we see that the limit
\begin{align*}
\lim_{\substack{r \to a \\ a \neq r \in \mathcal{R}_+}}b_r^{n, j}(f)\gamma_r^{-1}=\lim_{m \to \infty}\lim_{\substack{r \to a \\ a \neq r \in \mathcal{R}_+}}g_m^{n, j}(r)
\end{align*}
exists for each $0 \leq j \leq n$ and satisfies 
\begin{align*}
\binom{n+1}{j}\lim_{\substack{r \to a \\ a \neq r \in \mathcal{R}_+}}b_r^{n, 0}(f)\gamma_r^{-1}&=\lim_{m \to \infty} \binom{n+1}{j}\lim_{\substack{r \to a \\ a \neq r \in \mathcal{R}_+}}g_m^{n, 0}(r) \\
&=\lim_{m \to \infty}\lim_{\substack{r \to a \\ a \neq r \in \mathcal{R}_+}}g_m^{n, j}(r)=\lim_{\substack{r \to a \\ a \neq r \in \mathcal{R}_+}}b_r^{n, j}(f)\gamma_r^{-1}.
\end{align*}
Since 
\begin{align*}
|f-f_m|_{n+1}=\sup_{r \in \mathcal{R}}\{|g_m^{n, 0}(r)-b_r^{n, 0}(f)\gamma_r^{-1}|,\cdots , |g_m^{n, n}(r)-b_r^{n, n}(f)\gamma_r^{-1}|\}<\varepsilon
\end{align*}
for sufficiently large $m \in \mathbb{Z}_{>0}$, we conclude the proof.
\end{proof}

\begin{cor}
Let $n \geq 1$ and $f=\sum_{r \in \mathcal{R}}b_r(f)\chi_r \in C(R, K)$. The following conditions are equivalent.
\begin{enumerate}
\item $f \in C^n(R, K)$ and $f'=0$.
\item $\lim_{r \in \mathcal{R}_+} b_r(f)\gamma_r^{-n}=0$.
\end{enumerate}
\end{cor}

\begin{proof}
The condition (1) implies the condition (2) by Lemma \ref{half_corofmain}. To prove the converse, we suppose that $\lim_{r \in \mathcal{R}_+} b_r(f)\gamma_r^{-n}=0$. This means that for any $\varepsilon>0$ there exists a finite subset $S_{\varepsilon} \subset \mathcal{R}_+$ such that $|b_r(f)\gamma_r^{-n}|<\varepsilon$ for any $r \in \mathcal{R}_+ \setminus S_{\varepsilon}$. Since
\begin{align*}
|b_r(f)\gamma_r^{-1}|=|b_r(f)\gamma_r^{-n}\cdot \gamma_r^{n-1}|<q^{-(n-1)(l(r)-1)}\varepsilon \leq \varepsilon
\end{align*}
for any $r \in \mathcal{R}_+ \setminus S_{\varepsilon}$, it follows from Theorem \ref{gN^1} that $f \in C^1(R, K)$ and $f'=0$. Let $1 \leq k \leq n-1$ and suppose that $f \in C^k(R,K)$ and $f'=0$. Since 
\begin{align*}
|b_r(f)\gamma_r^{-k-1}|=|b_r(f)\gamma_r^{-n}\cdot \gamma_r^{n-k-1}|<q^{-(n-k-1)(l(r)-1)}\varepsilon\leq\varepsilon
\end{align*}
for any $r \in \mathcal{R}_+ \setminus S_{\varepsilon}$, the infinite sum
\begin{align*}
f=\sum_{r \in \mathcal{R}}b_r(f)\chi_r=\sum_{r \in \mathcal{R}}\frac{b_r(f)}{\gamma_r^{k+1}}\gamma_r^{k+1}\chi_r
\end{align*}
converges in the $K$-Banach space $(C^{k+1}(R, K), |\cdot|_{k+1})$. Thus, we see that $f \in C^{k+1}(R, K)$.
\end{proof}

\section{Norms on $C^n(R, K)$ and $n$-th Lipschitz functions}

The main purpose of this section is to prove Theorem \ref{samenorm} and Theorem \ref{charn-lip}. See Definition \ref{defn-lip} for the $n$-th Lipschitz functions.

\subsection{$n$-th Lipschitz functions}

\begin{lem}\label{inclu}
Let $n\geq1$.
\begin{enumerate}
\item A $C^{n}$-function is an $n$-th Lipschitz function.

\item An $n$-th Lipschitz function is a $C^{n-1}$-function. 
\end{enumerate}
\end{lem}

\begin{proof}

1. If $f$ is a $C^n$-function, then $\Phi_nf$ can be extended to a continuous function on $R^{n+1}$. Then $|\Phi_nf|$ is bounded by the compactness of $R^{n+1}$. Thus, $f$ is an $n$-th Lipschitz function.

2. If $f$ is an $n$-th Lipschitz function and $x_1, x_2, \cdots, x_{2n}$ are pairwise distinct, then
\begin{align*}
&\left|\Phi_{n-1}f(x_1,\cdots,x_{n})-\Phi_{n-1}f(x_{n+1},\cdots,x_{2n})\right|\\
=&\left|\sum_{j=1}^{n}\left(\Phi_{n-1}f(x_{n+1},\cdots x_{n+j-1},x_{j},\cdots,x_n)-\Phi_{n-1}(x_{n+1},\cdots,x_{n+j},x_{j+1},\cdots,x_n)\right)\right|\\
\leq&\max_{1\leq j\leq n}\left\{\left|\Phi_{n-1}f(x_{n+1},\cdots x_{n+j-1},x_{j},\cdots,x_n)-\Phi_{n-1}(x_{n+1},\cdots,x_{n+j},x_{j+1},\cdots,x_n)\right|\right\}\\
=&\max_{1\leq j\leq n}\left\{\left|\Phi_nf(x_{n+1},\cdots,x_{n+j},x_{j},\cdots,x_{n})\right|\cdot\left|x_j-x_{n+j}\right|\right\}\\
\leq&A_f\max_{1\leq j\leq n}\left\{\left|x_j-x_{n+j}\right|\right\}.
\end{align*}
Here we used (\ref{n-thquot}) in the second equality and the definition of $A_f$ (see Definition \ref{defn-lip}) in the second inequality. Therefore $\Phi_{n-1}f$ is uniformly continuous on $\bigtriangledown^{n}R$ (see (\ref{btd}) for the definition of $\bigtriangledown^{n}R$) and $f$ can be extended to a continuous function on $R^n$. Hence, $f$ is a $C^{n-1}$-function.
\end{proof}

By Lemma \ref{inclu}, we can expand $f$ like (\ref{expC^{n-1}})
as a $C^{n-1}$-function. In the following, we give a proof of Theorem \ref{charn-lip} and show that $Lip_n(R,K)$ is a $K$-Banach space.

\begin{proof}[Proof of Theorem \ref{charn-lip}]
Assume that $f$ is an $n$-th Lipschitz function. Note that $b_r^{n-1,j}(f)\gamma_r^{-1}=\psi_{n-1-j}D_jf(r,r_-)$ for all $r\in\mathcal{R}_{+}$, $0\leq j\leq n-1$ by Theorem \ref{main}. 
Therefore, we have
\begin{align*}
\left|b_r^{n-1,0}(f)\gamma_r^{-1}\right|
=\left|\psi_{n-1}D_0f(r,r_-)\right
|=\left|\Phi_{n}f(r,r_-,\cdots,r_-)\right|
\leq A_f
\end{align*}
and, for $1\leq j\leq n-1$,
\begin{align*}
\left|b_r^{n-1,j}(f)\gamma_r^{-1}\right|
=\left|\psi_{n-1-j}D_jf(r,r_-)\right|
\leq\max_{1\leq i\leq j+1}\left\{\left|\psi_{n-1,i}f(r,r_-)\right|\right\}
\leq A_f
\end{align*}
by Lemma \ref{lem78.3}.
It follows that
\begin{align*}
\sup_{\substack{r\in\mathcal{R}_{+}\\0\leq j\leq n-1}}\left\{\left|b_r^{n-1,j}(f)\gamma_r^{-1}\right|\right\}<\infty.
\end{align*}
To show the converse, we apply Corollary \ref{keycor} with $\varepsilon=\sup_{\substack{r\in\mathcal{R}_{+}\\0\leq j\leq n-1}}\{|b_r^{n-1,j}(f)\gamma_r^{-1}|\}$, $\delta>1$ and $a=c=0$.
\end{proof}

\begin{rmk}
Let $f$ be an $n$-th Lipschitz function. For expansion (\ref{expC^{n-1}})
\begin{align*}
\left|f\right|_{Lip_n}=\sup_{\substack{r\in\mathcal{R}\\0\leq j\leq n-1}}\left\{\left|b_r^{n-1,j}(f)\gamma_r^{-1}\right|\right\}
\end{align*}
is a norm of $Lip_n(R,K)$. We also denote it by $|f|_n$. (See (\ref{def of norm}).)
\end{rmk}

The following proof is adopted from {\cite[Corollary 3.2]{dS16}}. 

\begin{prop}\label{BanachLip_n}
Let $n\geq1$. $Lip_n(R,K)$ is a $K$-Banach space with respect to the norm $|\cdot|_n$.
\end{prop}

\begin{proof}
Theorem \ref{charn-lip} shows that the correspondence
\begin{align*}
Lip_n(R,K)\to \left(l^{\infty}(\mathcal{R})\right)^n;\sum_{r\in\mathcal{R}}\sum_{j=0}^{n-1}b_{r}^{n-1,j}(f)\gamma_r^{n-1-j}(x-r)^j\chi_r\mapsto\left[(b_r^{n-1,j}(f)\gamma_r^{-1})_{r}\right]_{0\leq j\leq n-1}
\end{align*}
is a norm-preserving isomorphism of $Lip_n(R,K)$ with the Banach space $\left(l^{\infty}(\mathcal{R})\right)^n$ of direct product of all bounded functions on $\mathcal{R}$. Thus, $Lip_n(R,K)$ is complete.
\end{proof}

\subsection{Proof of Theorem \ref{samenorm}}

\begin{lem}
Let $n\geq1$ and $f\in C^n(R,K)$. We expand 
\begin{align*}
f=&\sum_{r\in\mathcal{R}}\sum_{j=0}^{n-1}b_{r}^{n-1,j}(f)\gamma_r^{n-1-j}(x-r)^j\chi_r\in C^{n-1}(R,K)
\end{align*}
and
\begin{align*}
f=&\sum_{r\in\mathcal{R}}\sum_{j=0}^{n}b_{r}^{n,j}(f)\gamma_r^{n-j}(x-r)^j\chi_r\in C^n(R,K)
\end{align*}
as elements of $C^{n-1}(R,K)$ and $C^n(R,K)$ respectively. Then for all $r\in\mathcal{R}_{+}$ and $0\leq j\leq n-1$, we have
\begin{align*}
b_{r}^{n-1,j}(f)\gamma_r^{-1}=b_r^{n,j}(f)+\dbinom{n}{j}\sum_{\substack{r'\in\mathcal{R}\\r'\triangleleft\ r_-}}b_{r'}^{n,n}(f),
\end{align*}
where $\lhd$ is defined in (\ref{deflhd}).
\end{lem}

\begin{proof}
Noting that $D_1\chi_r=0$ and $D_n(x^n)=1$, we have
\begin{align*}
D_nf(x)=\sum_{r\in\mathcal{R}}b_r^{n,n}(f)\chi_r(x).
\end{align*}
The definition (\ref{defpsi_jf2}) of $\psi_nf$ shows 
\begin{align*}
\psi_{n-1}f(x,y)=(x-y)\psi_{n}f(x,y)+D_nf(y).
\end{align*}
Hence, we have
\begin{align*}
b_r^{n-1,j}(f)\gamma_r^{-1}&=\psi_{n-1-j}D_jf(r,r_-)\\
&=\gamma_r\psi_{n-j}D_jf(r,r_-)+D_{n-j}D_jf(r_-)\\
&=b_r^{n,j}(f)+\dbinom{n}{j}D_nf(r_-)\\
&=b_r^{n,j}(f)+\dbinom{n}{j}\sum_{\substack{r'\in\mathcal{R}\\r'\triangleleft\ r_-}}b_{r'}^{n,n}(f).
\end{align*}
\end{proof}

\begin{cor}\label{normineq}
Let $n\geq1$ and $f\in C^{n+1}(R,K)$. Then $|f|_{n}\leq|f|_{n+1}$ holds.
\end{cor}

\begin{proof}
We have that
\begin{align*}
|f|_{n}&=\sup_{\substack{r\in\mathcal{R}\\0\leq j\leq n-1}}\left\{\left|b_r^{n-1,j}(f)\gamma_r^{-1}\right|\right\}\\
&=\sup_{\substack{r\in\mathcal{R}_{+}\\0\leq j\leq n-1}}\left\{\left|D_jf(0)\right|, \left|b_r^{n,j}(f)+\dbinom{n}{j}\sum_{\substack{r'\in\mathcal{R}\\r' \triangleleft\ r_-}}b_{r'}^{n,n}(f)\right|\right\}\\
&\leq\sup_{\substack{r\in\mathcal{R}_{+}, r'\in\mathcal{R}\\0\leq j\leq n-1}}\left\{\left|D_jf(0)\right|, \left|b_r^{n,j}(f)\right|, \left|b_{r'}^{n,n}(f)\right|\right\}\\
&=\sup_{\substack{r\in\mathcal{R}_{+}\\0\leq j\leq n}}\left\{\left|D_jf(0)\right|, \left|b_r^{n,j}(f)\right|\right\}\\
&\leq\sup_{\substack{r\in\mathcal{R}_{+}\\0\leq j\leq n}}\left\{\left|D_jf(0)\right|, \left|b_r^{n,j}(f)\gamma_r^{-1}\right|\right\}\\
&=|f|_{n+1}.
\end{align*}
Here, the first equality follows from the definition (\ref{def of norm})  of $|f|_n$ and second and third equalities follow from Theorem \ref{main}.
\end{proof}

\begin{proof}[Proof of Theorem \ref{samenorm}]
We will prove by induction on $n$. First, if we expand $f=\sum_{r\in\mathcal{R}}b_r(f)\chi_r\in C^1(R,K)$, then
\begin{align*}
\left|f\right|_{C^1}&=\max\left\{\left|f\right|_{\rm{sup}}, \left|\Phi_1f\right|_{\rm{sup}}\right\}\\
&=\max\left\{\sup_{r\in\mathcal{R}}\left\{\left|b_r(f)\right|\right\}, \sup_{r\in\mathcal{R}_{+}}\left\{\left|b_r(f)\gamma_r^{-1}\right|\right\}\right\}\\
&=\sup_{r\in\mathcal{R}}\left\{\left|b_r(f)\gamma_r^{-1}\right|\right\}\\
&=\left|f\right|_1
\end{align*}
by Theorem \ref{charn-lip}. Next, assume $|f|_{C^{n-1}}=|f|_{n-1}$ for all $f\in C^{n-1}(R,K)$. If we expand $f=\sum_{r\in\mathcal{R}}\sum_{j=0}^{n}b_{r}^{n,j}(f)\gamma_r^{n-j}(x-r)^j\chi_r\in C^n(R,K)$, then
\begin{align}\label{|f|_C^n}
\left|f\right|_{C^{n}}&=\max_{0\leq k\leq n}\left\{\left|\Phi_kf\right|_{\rm{sup}}\right\}\notag\\
&=\max\left\{\left|f\right|_{n-1}, \sup_{\substack{r\in\mathcal{R}_{+}\\0\leq j\leq n-1}}\left|b_r^{n-1,j}(f)\gamma_r^{-1}\right|\right\}.
\end{align}
Now, Corollary \ref{normineq} and the definition of $|f|_n$ imply (\ref{|f|_C^n}) $\leq|f|_n$. On the other hand, if $f\in C^k(R,K)$, we have $|D_{k-1}f|_{\rm{sup}}\leq|f|_{k-1}$by Lemma \ref{sup_k}, so
\begin{align*}
|f|_n&=\sup_{\substack{r\in\mathcal{R}_{+}\\0\leq j\leq n-1}}\left\{\left|D_{j}f(0)\right|, \left|b_r^{n-1,j}(f)\gamma_r^{-1}\right|\right\}\leq\text{(\ref{|f|_C^n})}.
\end{align*}
As a result, we obtain $|f|_{C^n}=|f|_n$ for all $n\geq1$.
\end{proof}

\section{characterizations of various functions}
In this section, we will show characterizations of three kinds of  functions in terms of the wavelet coefficients.

\begin{dfn}[{\cite[Section 24]{Sc84}}]
\begin{enumerate}
\item An element $x$ of $K$ is called positive if $|1-x|<1$. The group of all positive elements of $K$ is denoted by $K^{+}$. Let
\begin{align*}
{\rm sgn}:K^{\times}\to \Sigma:=K^{\times}/K^{+}
\end{align*}
be the canonical homomorphism.
\item Let $\alpha\in\Sigma$. A function $f:R\to K$ is called monotone of type $\alpha$ if for all $x,y\in R$, $x\neq y$, we have $f(x)\neq f(y)$ and
\begin{align*}
{\rm sgn}(f(x)-f(y))=\alpha\ {\rm sgn}(x-y).
\end{align*}
Such an $f$ is called increasing if $\alpha$ is the identity element of $\Sigma$.
\end{enumerate}
\end{dfn}

\begin{prop}[{\cite[Section 24]{Sc84}}]\label{incre}
Let $f:R\to K$. Then $f$ is increasing if and only if for all $x,y\in R$, $x\neq y$
\begin{align*}
\left|\frac{f(x)-f(y)}{x-y}-1\right|<1.
\end{align*}
More generally, if there exists an element $s\in K^{\times}$ such that for all $x,y\in R$, $x\neq y$
\begin{align*}
\left|s^{-1}\frac{f(x)-f(y)}{x-y}-1\right|<1,
\end{align*}
then $f$ is monotone of type $\alpha$ where $\alpha={\rm sgn}\ s$.
\end{prop}

\begin{prop}\label{charincre}
Let $f=\sum_{r\in\mathcal{R}}b_r(f)\chi_r\in C(R,K)$. Then $f$ is increasing if and only if $|b_r(f)\gamma_r^{-1}-1|<1$ for each $r\in\mathcal{R}_{+}$. Generally, $f$ is monotone of type $\alpha$ if and only if $|s^{-1}b_r(f)\gamma_r^{-1}-1|<1$ for each $r\in\mathcal{R}_{+}$ where ${\rm sgn} (s)=\alpha$.
\end{prop}

\begin{proof}
If $f$ is monotone of type $\alpha$ and ${\rm sgn}(s)=\alpha$, $s^{-1}f$ is increasing, hence we may assume that $f$ is increasing. Then, Proposition \ref{incre} and the fact that $f(r)-f(r_-)=b_r(f)$ for all $r\in\mathcal{R}_{+}$ show that $|\Phi_1f(r,r_-)-1|=|b_r(f)\gamma_r^{-1}-1|<1$ for all $r\in\mathcal{R}_{+}$.
Conversely, if $\Phi_1f(r,r_-)$ is positive for all $r\in\mathcal{R}_{+}$, then $f$ is increasing by Lemma \ref{keylem1}.
\end{proof}

We recover \cite[Exercise 63.E]{Sc84} as a special case of Proposition \ref{charincre}.

\begin{dfn}[{\cite[Definition 76.6]{Sc84}}]
A function $f:R\to K$ is called pseudocontraction if $|\Phi_1f(x,y)|<1$ for all $x,y\in R$, $x\neq y$.
\end{dfn}

\begin{cor}\label{pseudocontra}
Let $f=\sum_{r\in\mathcal{R}}b_r(f)\chi_r\in C(R,K)$. A function $f$ is a pseudocontraction if and only if $|b_r(f)\gamma_r^{-1}|<1$ for all $r\in\mathcal{R}^{+}$.
\end{cor}

\begin{proof}
A pseudocontraction $f$ is Lipschitz and its Lipschitz constant $A_f$ is less than $1$. Thus, Theorem \ref{charn-lip} implies $\sup_{r\in\mathcal{R}_{+}}\{|b_r(f)\gamma_r^{-1}|\}<1$. Conversely, if $\sup_{r\in\mathcal{R}_{+}}\{|b_r(f)\gamma_r^{-1}|\}<1$, $f$ is Lipschitz by Theorem \ref{charn-lip}. Hence $f$ is a pseudocontraction.
\end{proof}

A function $f$ is called an isometry if $|\Phi_1f(x,y)|=1$ for all $x,y\in R$, $x\neq y$. For example, an increasing function is an isometry by Proposition \ref{incre}.

\begin{prop}\label{iso}
Let $f=\sum_{r\in\mathcal{R}}b_r(f)\chi_r\in C(R,K)$. A function $f$ is an isometry if and only if $|b_r(f)\gamma_r^{-1}|=1$ for all $r\in\mathcal{R}^{+}$ and $|b_{r_1}(f)-b_{r_2}(f)|=|\pi^{l(r_1)-1}|$ for all $r_1, r_2\in\mathcal{R}^{+}$ with $r_1\neq r_2$, $({r_1})_-=({r_2})_-$.
\end{prop}

\begin{proof}
If $f$ is an isometry, for all $r\in\mathcal{R}^{+}$
\begin{align*}
\left|b_r(f)\right|=\left|f(r)-f(r_-)\right|=\left|r-r_-\right|=\left|\gamma_r\right|.
\end{align*}
For all $r_1, r_2\in\mathcal{R}^{+}$ with $r_1\neq r_2$, $({r_1})_-=({r_2})_-$, 
\begin{align*}
\left|b_{r_1}(f)-b_{r_2}(f)\right|=\left|f(r_1)-f(r_2)\right|=\left|r_1-r_2\right|=\left|\pi^{l(r_1)-1}\right|.
\end{align*}
The next step is to show sufficiency. Let $x,y\in R$, $x\neq y$ and 
\begin{align*}
x&=c_0+c_1\pi+\cdots\\
y&=d_0+d_1\pi+\cdots.
\end{align*}
Now take $m\geq0$ which is $|x-y|=|\pi^m|$ (that is, $c_0=d_0, \cdots, c_{m-1}=d_{m-1}, c_m\neq d_m$). Then
\begin{align*}
f(x)-f(y)\equiv 
\begin{cases}
b_{c_0+\cdots+c_m\pi^m}(f)-b_{d_0+\cdots+d_m\pi^m}(f)&(c_m, d_m\neq0)\\
b_{c_0+\cdots+c_m\pi^m}(f)&(d_m=0)\\
b_{d_0+\cdots+d_m\pi^m}(f)&(c_m=0)
\end{cases}
\bmod{\pi^{m+1}}.
\end{align*}
In all cases, $|f(x)-f(y)|=|\pi^m|$, so $f$ is an isometry.
\end{proof}

The following example illustrates why the second condition is required in Proposition \ref{iso}.

\begin{ex}
Let $p$ be an odd prime, $R=\mathbb{Z}_p$, $\mathcal{R}=\mathbb{Z}$, and
\begin{align*}
f(x)=x+\chi_1=\sum_{r\in\mathcal{R}}\gamma_r\chi_r+\chi_1=2\chi_1+2\chi_2+3\chi_3+\cdots.
\end{align*}
We see that $f$ is not an isometry but $|b_r(f)\gamma_r^{-1}|=1$ for all $r\in\mathcal{R}^{+}$. On the other hand, 
\begin{align*}
g(x)=x+\chi_1-\chi_2=\sum_{r\in\mathcal{R}}\gamma_r\chi_r+\chi_1-\chi_2=2\chi_1+\chi_2+3\chi_3+\cdots
\end{align*}
is an isometry but is not increasing.
\end{ex}

\vspace{10pt}
\noindent
Mathematical Institute, Graduate School of Science, Tohoku University,\\
6-3 Aramakiaza, Aoba, Sendai, Miyagi 980-8578, Japan.\\
E-mail address: \textbf{hiroki.ando.s8@dc.tohoku.ac.jp}\\

\noindent
Mathematical Institute, Graduate School of Science, Tohoku University,\\
6-3 Aramakiaza, Aoba, Sendai, Miyagi 980-8578, Japan.\\
E-mail address: \textbf{yu.katagiri.s3@dc.tohoku.ac.jp}

\end{document}